\documentclass[final,3p]{elsarticle}
\usepackage{eurosym}
\usepackage{amsfonts}
\usepackage{amsmath}
\usepackage{amssymb}
\usepackage{graphicx}
\usepackage[tableposition=top]{caption}
\usepackage{subcaption}
\usepackage{float}
\usepackage{setspace}
\usepackage{placeins}
\usepackage{mathrsfs}
\usepackage{color}
\usepackage{booktabs}
\usepackage{hyperref}

\newsavebox \foobox
\newlength{\foodim}

\graphicspath{{C:/Figures/}}
\newtheorem{theorem}{Theorem}
\newtheorem{proof}{Proof}

\newtheorem{lemma}{Lemma}

\newtheorem{proposition}{Proposition}
\newtheorem{remark}{Remark}

\journal{...........}

\begin{document}

\begin{frontmatter}
\title{Global and local asymptotic stability of an epidemic reaction-diffusion model with a nonlinear incidence}
\author{Lamia Djebara $^{a}$, Redouane Douaifia $^{b}$, Salem Abdelmalek $^{c}$, Samir Bendoukha $^{d,*}$}
\address{(a)Department of Mathematics and Computer Science, ICOSI Laboratory, University of Khenchela,Khenchela, 04000, Algeria.\\
(b)Laboratory of Mathematics, Informatics and Systems (LAMIS), Larbi Tebessi University - Tebessa, Algeria\\
(c)Department of Mathematics and Computer Science, Larbi Tebessi University - Tebessa, Algeria\\
(d)Department of Electrical Engineering, College of Engineering, Yanbu, Taibah University, Saudi Arabia.\\
(*)Corresponding author, email: sbendoukha@taibahu.edu.sa}
\begin{abstract}
The aim of this paper is to study the dynamics of a reaction--diffusion SIS (susceptible-infectious-susceptible) epidemic model  with a nonlinear incidence rate describing the transmission of a communicable disease between individuals. We prove that the proposed model has two steady states under one condition. By analyzing the eigenvalues and using the Routh--Hurwitz criterion and an appropriately constructed Lyapunov functional, we establish the local and global asymptotic stability of the non negative constant steady states subject to the basic reproduction number being greater than unity and of the disease--free equilibrium subject to the basic reproduction number being smaller than or equal to unity in ODE case. By applying an appropriately constructed Lyapunov functional, we identify the condition of the global stability in the PDE case. Finally, we present some numerical examples illustrating and confirming the analytical results obtained throughout the paper.

\end{abstract}

\end{frontmatter}


\section{Introduction\label{SecModel}}

In this manuscript, we consider the reaction--diffusion epidemic phenomenon
proposed in \cite{Djebara2019}, which is an extended version of the SIS
epidemic model with the nonlinear incidence $u\varphi \left( v\right) $. The
system is described as%
\begin{equation}
\left \{
\begin{array}{c}
\dfrac{\partial u}{\partial t}-d_{1}\Delta u=\Lambda -\mu u-\lambda u\varphi
\left( v\right) :=F\left( u,v\right) \text{ \  \  \ in }\left( 0,\infty
\right) \times \Omega , \\
\dfrac{\partial v}{\partial t}-d_{2}\Delta v=-\sigma v+\lambda u\varphi
\left( v\right) :=G\left( u,v\right) \text{ \  \  \  \  \  \ in }\left( 0,\infty
\right) \times \Omega .%
\end{array}%
\right.   \label{e1.1}
\end{equation}%
Throughout this paper, the notation $\Delta $ denotes the Laplacian
operators on $\Omega $, where $\Omega $ is an open bounded subset of $%
\mathbb{R}^{n}$ with smooth boundary $\partial \Omega $. The constant
parameters $d_{1},d_{2}\geq 0$ are the diffusion coefficients. We assume the
initial conditions
\begin{equation}
u_{0}(x)=u(x,0),\ v_{0}(x)=v(x,0)\text{\  \  \  \  \  \  \  \ in }\Omega ,
\label{e1.2}
\end{equation}%
where $u_{0},v_{0}\in C(\overline{\Omega })$, and impose homogeneous Neumann
boundary conditions%
\begin{equation}
\frac{\partial u}{\partial \nu }=\frac{\partial v}{\partial \nu }=0\text{\  \
\  \  \  \  \  \  \  \  \  \  \  \ on }(0,\infty )\times \partial \Omega ,  \label{e1.3}
\end{equation}%
with $\nu $ being the unit outer normal to $\partial \Omega $. We will also
assume that the initial conditions $u_{0}(x),v_{0}(x)\in \mathbb{R}_{\geq 0}$%
.

The considered system (\ref{e1.1})--(\ref{e1.3}) may describe the
transmission of a communicable disease between individuals such as HIV/AIDS.
The population assumed in the model is divided into two classes, susceptible
and infective \cite{hama}. Functions $u(x,t)$ and $v(x,t)$\ denote the non
dimensional population densities of the susceptible and infective
individuals at location $x$ and time $t$, respectively. The constant
parameter $\Lambda >0$\ represents the average growth of different
categories of susceptibles whether through birth or migration, sometimes
referred to as the recruitment rate of the population. In addition, $\mu $
denotes the natural death rate, $\lambda $ is the rate at which susceptibles
turn into infectives, and $\sigma $ is the rate at which invectives recover
from the disease. For the purpose of this study, we will assume that $\mu
,\sigma ,\lambda >0$. The incidence function $\varphi (v)$ introduces a
nonlinear relation between the two classes of individuals. We assume $%
\varphi $ to be a continuously differentiable function on $\mathbb{R}^{+}$
satisfying%
\begin{equation}
\varphi \left( 0\right) =0,  \label{e1.4}
\end{equation}%
and%
\begin{equation}
0<v\varphi ^{\prime }\left( v\right) \leq \varphi \left( v\right) \text{ \
for all }v>0.  \label{e1.5}
\end{equation}

Infectious diseases are the leading cause of death in all living organisms.
The study of epidemiology has attracted the attention of a vast number of
researchers with the aim of improving the treatment of these diseases
through planning and predictions of the spread of the disease thereby
reducing mortality rates. In \cite{Hethcote1976}, Hethcote studied a simple
model resulting from bilinear and standard incidences $u\varphi \left(
v\right) =\lambda uv$ with\ $d_{1}=d_{2}=0$ and proved that the asymptotic
stability regions for the equilibrium points depend on the basic
reproduction number. Similar works and results can be found in \cite%
{Brauer2001,Korobeinikov2002,MA2003,Abha2015}. Nonlinear incidences of the
form $u\varphi \left( v\right) $ were investigated in\  \cite%
{Capasso1978,Alexander2004,Li2016,Jia2017,Avila2017} among numerous other
studies. However, these studies only considered the simple diffusion--free
case, which disregards the distribution of individuals within their
respective spatial domain. In \cite{Capasso1978}, the authors proposed a
mathematical epidemic model with a saturated incidence $S\varphi \left(
I\right) $, which they used to study the spread of cholera in Bari. The
study investigated the positivity, global existence, and uniqueness of the
system solution as well as the asymptotic stability of the equilibrium
points using an extension of the threshold theory. The authors considered as
a numerical example the case%
\begin{equation*}
\varphi \left( I\right) =\frac{kI}{1+\left( \frac{I}{\alpha }\right) }%
,k>0,\alpha >0.
\end{equation*}

In \cite{Webb1981}, Webb established the existence of solutions $\left(
S\left( t,.\right) ,I\left( t,.\right) ,R\left( t\right) \right) $ for a
spatially--diffusive SIR model with $\Lambda =\mu =0$,$\ d_{1}=d_{2}$,\ and $%
\varphi \left( v\right) =\lambda v$ in the one dimensional ($n=1$) spatial
region $\left[ -L,L\right] $. Using the theory of linear and nonlinear
semi-linear groups in Banach distances and Lyapunov stability techniques for
dynamic systems in metric spaces, the authors analyzed the behavior of
solutions as time goes to infinity and showed that $\left( S\left(
t,.\right) ,I\left( t,.\right) ,R\left( t,.\right) \right) \rightarrow
\left( S_{\infty },0,R_{\infty }\right) $ with $S_{\infty },R_{\infty }$
being positive constant functions on $\left[ -L,L\right] $. The same
scenario was considered again in section three of \cite{kim2010}\ but with $%
\Lambda ,\mu >0$. Considering a bird system with homogeneous Neumann
boundary conditions, the authors investigated the uniform bounds of the
solutions and the local and global stability of the equilibria. For the
purpose pf establishing the global stability the Lyapunov function
\begin{equation*}
V\left( t\right) =\int_{\Omega }\left[ \frac{1}{2}\left( S-\frac{\Lambda }{%
\mu }\right) ^{2}+\frac{\Lambda }{\mu }I\right] dx
\end{equation*}%
was used. A similar investigation was carried out in \cite{Chinviriyasit2010}%
.

A criss--cross infection model describing the spread of FIV (Feline
immunodeficiency virus) was proposed and studied in \cite{Fitzgibbon1995}.
This model falls under (\ref{e1.1}) with $n\geq 1$, $d_{i}>0$, $\Lambda =\mu
=0$, and bilinear incidence $u\varphi \left( v\right) =\lambda uv$. Another
important study is \cite{Alexander2004}, where the local and global
asymptotic stability of the simple epidemiological model with $\varphi
\left( v\right) =\beta \left[ 1+\varphi _{\eta }\left( v\right) \right] v$
were investigated by means of the Poincar\'{e} index theory. The study
showed that the basic reproductive number is independent of the functional $%
\varphi _{\eta }\left( v\right) $. Lyapunov functions were proposed in \cite%
{Korobeinikov2005} for ODE SIS and SEIR models with incidence rate $\varphi
\left( v\right) g\left( u\right) $ and $\Lambda =\mu \geq 0$. The same
authors generalized the incidence rate to the very general form $f\left(
u,v\right) $ and $\Lambda =\mu \geq 0$ in \cite{Korobeinikov2006}. Using
extended Lyapunov functions based on those constructed in \cite%
{Korobeinikov2005}, they established the global asymptotic stability.

More sophisticated nonlinear incidences of the form $k\frac{vu}{1+\delta v}$
were\ studied in \cite{Lahrouz2011,Che2014}, where\ the authors considered
an Avian influenza model and proved that the global asymptotic stablility
depends on the basic reproductive number. A more general Lyapunov function
was developed in \cite{Li2016} and used\ to establish the existence of an
endemic equilibrium for ODE SIS models with the nonlinear incidence $\varphi
\left( v\right) =\lambda u\varphi \left( v\right) $. Furthermore, a
spatially diffusive SIR epidemic model with two scenarios $\left(
d_{1}=0,d_{2}>0\right) $ and $\left( d_{1}>0,d_{2}=0\right) $ was assumed in
\cite{Kuniya2016} with all parameters being spatially heterogeneous, where $%
\varphi \left( v\right) =\lambda \left( x\right) v$ and $\sigma \left(
.\right) ,$ $\mu \left( .\right) ,$ $\gamma \left( .\right) ,$ $\lambda
\left( .\right) \in C\left( \overline{\Omega };\mathbb{R}\right) $ are all
strictly positive. The global asymptotic stability of $E_{0}=\left( \frac{%
b\left( x\right) }{\beta \left( x\right) },0\right) $ and $E^{\ast }=\left(
S^{\ast }\left( x\right) ,I^{\ast }\left( x\right) \right) $ was studied and
the authors concluded that the standard threshold dynamical behaviors depend
on the basic reproductive number. More recently, the authors of \cite%
{Jia2017} investigated the stability of the disease--free equilibrium and
the unique endemic equilibrium for an ODE SITAR epidemiological model with
general nonlinear incidence rate $u\varphi \left( v\right) $ a using the
geometric approach.

In the present study, we study the existence of equilibria and their
asymptotic stability conditions for the diffusive epidemic model considered
in \cite{Djebara2019}, which is an extension of that proposed in \cite%
{Li2016}. We recall that in \cite{Djebara2019}, we established the global
existence of solutions to problem (\ref{e1.1})--(\ref{e1.3}). We have now
established the system model and parameter descriptions. Section \ref%
{SecSteadyState} will define the basic reproductive number $R_{0}$ of the
proposed model and establish the existence of two equilibria. The local
asymptotic stability and instability of the disease--free equilibrium and
the endemic equilibrium are investigated. Section \ref{SecGlobal} will prove
that the two steady states of the model are globally asymptotically stable
using an appropriate Lyapunov functional. Finally, Section \ref{SecExampl}
will present some numerical examples to validate the theoretical analysis
presented throughout the paper.

\section{Preliminary Properties of the Model\label{SecSteadyState}}

Let us assume that the initial conditions $\left( u_{0},v_{0}\right) \in
\mathbb{R}_{\geq 0}^{2}$. Note that for $(u,v)\in \mathbb{R}_{\geq 0}^{2}$,
we have%
\begin{equation*}
\left \{
\begin{array}{l}
F(0,v)=\Lambda >0, \\
G(u,0)=0,%
\end{array}%
\right.
\end{equation*}
which makes the function $(F,G)^{T}$ essentially nonnegative. Hence, the
nonnegative quadrant $\mathbb{R}_{\geq 0}^{2}$ is an invariant set, see \cite%
[Proposition 2.1]{Haddad2010,Quitner2007}. By dropping the spatial variable,
the proposed system reduces to the following system of ordinary differential
equations (ODE):%
\begin{equation}
\left \{
\begin{array}{ll}
\frac{du}{dt}=F\left( u,v\right) & \  \  \  \  \text{in }(0,\infty ), \\
\frac{dv}{dt}=G\left( u,v\right) & \  \  \  \  \text{in }(0,\infty ),%
\end{array}%
\right.  \label{e1.6}
\end{equation}%
with initial conditions%
\begin{equation}
u\left( 0\right) =u_{0}\geq 0,\  \  \ v\left( 0\right) =v_{0}\geq 0.
\label{e1.7}
\end{equation}%
In the following subsections, we define an invariant region for the system,
identify the system's equilibria and their relation to the basic
reproduction number $R_{0}$, establish the global existence of solutions in
time, and investigate the local stability of the system in the ODE and PDE\
scenarios.

\subsection{Invariant Regions}

Throughout this paper,we let\ $N=u+v$ and $\sigma _{0}=\min \left( \sigma
,\mu \right) $. We also define the region%
\begin{equation*}
D=\left \{ (u,v):u,v\geq 0\text{ \ and \ }u+v\leq \frac{\Lambda }{\sigma _{0}%
}\right \} .
\end{equation*}%
The following proposition shows that $D$ is an invariant region of system (%
\ref{e1.6})--(\ref{e1.7}).

\begin{proposition}
The region $D$ is non--empty, attracting and positively invariant.
\end{proposition}

\begin{proof}
We start by summing the equations of system (\ref{e1.6})--(\ref{e1.7}),
which yields%
\begin{equation*}
\frac{d}{dt}N\left( t\right) =u_{t}+v_{t}\leq \Lambda -\sigma _{0}N.
\end{equation*}%
This implies that%
\begin{equation*}
N\left( t\right) \leq \frac{\Lambda }{\sigma _{0}}\left( 1-e^{-\sigma
_{0}t}\right) +N_{0}e^{-\sigma _{0}t}.
\end{equation*}%
Substituting the value of $N$ yields%
\begin{equation*}
\left( u+v\right) \left( t\right) \leq \frac{\Lambda }{\sigma _{0}}\left(
1-e^{-\sigma _{0}t}\right) +\left( u+v\right) \left( 0\right) e^{-\sigma
_{0}t},
\end{equation*}%
for $t\geq 0$. If the initial states satisfy $\left( u+v\right) \left(
0\right) \leq \frac{\Lambda }{\sigma _{0}}$, then $\left( u+v\right) \left(
t\right) \leq \frac{\Lambda }{\sigma _{0}}$ and%
\begin{equation*}
\limsup_{t\rightarrow \infty }N\left( t\right) \leq \frac{\Lambda }{\sigma
_{0}}.
\end{equation*}%
As a result, region $D$ is positively invariant and attracting within $%
\mathbb{R}_{\geq 0}^{2}$. Therefore, it is sufficient to consider the
dynamics of the model within $D$ as $D$ is the biologically feasible region
of the system where the existence and uniqueness results hold for the system.
\end{proof}

\subsection{Existence of Equilibria and the Basic Reproduction Number $R_{0}$%
}

In this section, we aim to show the existence of equilibrium solutions for (%
\ref{e1.1})--(\ref{e1.2}) and calculate the basic reproduction number.

\begin{theorem}
System (\ref{e1.6})--(\ref{e1.7}) has a single disease--free equilibrium $%
E_{0}=\left( \frac{\Lambda }{\mu },0\right) $. If the basic reproduction
number $R_{0}=\frac{\Lambda \lambda }{\mu \sigma }\varphi ^{\prime }(0)>1$,
then the system admits two distinct equilibria: $E_{0}$ and the positive
endemic equilibrium $E^{\ast }=\left( u^{\ast },v^{\ast }\right) \in \mathbb{%
R}_{>0}^{2}$.
\end{theorem}

\begin{proof}
The positive equilibria of model (\ref{e1.6})--(\ref{e1.7}) satisfy%
\begin{equation}
\left \{
\begin{array}{l}
F\left( u,v\right) =\Lambda -\lambda u\varphi \left( v\right) -\mu u=0, \\
G\left( u,v\right) =\lambda u\varphi \left( v\right) -\sigma v=0.%
\end{array}%
\right.  \label{e1.8}
\end{equation}%
If $u=0$, it is easy to see that the system has no equilibrium. On the other
hand, only equilibrium is found for $v=0$ and that is $E_{0}=\left( \frac{%
\Lambda }{\mu },0\right) $.

Next, we study the existence conditions of an endemic steady state in the
case $v>0$. From the second part of (\ref{e1.8}), and because $\lambda >0$
and $\varphi \left( v\right) >0$, we obtain%
\begin{equation*}
u=\frac{\sigma v}{\lambda \varphi \left( v\right) }.
\end{equation*}%
Substituting this into the first equation yields%
\begin{equation}
h\left( v\right) =0\text{ \ for any }v>0,  \label{e1.10}
\end{equation}%
where%
\begin{equation*}
h\left( v\right) =\frac{\Lambda \lambda }{\mu \sigma }\frac{\varphi \left(
v\right) }{v}-\frac{\sigma \lambda }{\mu \sigma }\varphi \left( v\right) -1
\end{equation*}%
is continuous for any $v>0$. It is clear that%
\begin{equation*}
\underset{v\rightarrow 0}{\lim }h\left( v\right) =\underset{v\rightarrow 0}{%
\lim }\frac{\Lambda \lambda }{\mu \sigma }\varphi ^{\prime }(v)-1=R_{0}-1,
\end{equation*}%
where $R_{0}$ is the basic reproduction number to be identified next. Using $%
\sigma _{0}=\min \left( \sigma ,\mu \right) $, we have%
\begin{equation*}
\underset{v\rightarrow \frac{\Lambda }{\sigma _{0}}}{\lim }h\left( v\right)
=h(\frac{\Lambda }{\sigma _{0}})=\frac{\lambda \left( \sigma _{0}-\sigma
\right) }{\mu \sigma }\varphi \left( \frac{\Lambda }{\sigma _{0}}\right)
-1<0.
\end{equation*}%
Hence for all\ $R_{0}>1$,%
\begin{equation*}
h(\frac{\Lambda }{\sigma _{0}})\underset{v\rightarrow 0}{\lim }h\left(
v\right) =h(\frac{\Lambda }{\sigma _{0}})(R_{0}-1)<0.
\end{equation*}%
By applying the intermediate value theorem, there exists a real $\ v^{\ast
}\in \left( 0,\frac{\Lambda }{\sigma _{0}}\right) $ such that (\ref{e1.10})
holds. Using (\ref{e1.5}), we can show that the function $h$ is
monotonically decreasing for all $v>0$ as%
\begin{equation*}
\frac{dh}{dv}\left( v\right) =\frac{\Lambda \lambda \left[ v\varphi ^{\prime
}\left( v\right) -\varphi \left( v\right) \right] -\sigma \lambda
v^{2}\varphi ^{\prime }\left( v\right) }{\mu \sigma v^{2}}<0.
\end{equation*}%
Then, there exists a unique real $v^{\ast }\in \left( 0,\frac{\Lambda }{%
\sigma _{0}}\right) $ such that $h(v^{\ast })=0$, which implies the
existence of a unique $u^{\ast }=\frac{\sigma v^{\ast }}{\lambda \varphi
\left( v^{\ast }\right) }$. Note that in $\left( \frac{\Lambda }{\sigma _{0}}%
,+\infty \right) $ the second equation of (\ref{e1.8}) has no solution
because
\begin{equation*}
\underset{v\in \left( \frac{\Lambda }{\sigma _{0}},+\infty \right) }{\max }%
h\left( v\right) \leq h\left( \frac{\Lambda }{\sigma _{0}}\right) <0.
\end{equation*}%
This concludes the proof.
\end{proof}

In the previous proof, we used the reproduction number $R_{0}$ of system (%
\ref{e1.6})--(\ref{e1.7}). It is now time to identify its value by means of
the next generation matrix method formulated in \cite{Sheikh2011} and
described further in \cite[Lemma 1 page 32]{Driessche2002}. System (\ref%
{e1.6})--(\ref{e1.7}) may be rewritten in vector form as%
\begin{eqnarray*}
\left(
\begin{array}{c}
v_{t} \\
u_{t}%
\end{array}%
\right) &=&\left(
\begin{array}{c}
\lambda u\varphi \left( v\right) -\sigma v \\
\Lambda -\lambda u\varphi \left( v\right) -\mu u%
\end{array}%
\right) \\
&=&%
\begin{pmatrix}
\lambda u\varphi \left( v\right) \\
0%
\end{pmatrix}%
-%
\begin{pmatrix}
\sigma v \\
-\Lambda +\lambda u\varphi \left( v\right) +\mu u%
\end{pmatrix}%
.
\end{eqnarray*}%
The Jacobian matrices corresponding to vectors $%
\begin{pmatrix}
\lambda u\varphi \left( v\right) \\
0%
\end{pmatrix}%
$ and $%
\begin{pmatrix}
\sigma v \\
-\Lambda +\lambda u\varphi \left( v\right) +\mu u%
\end{pmatrix}%
$ at the disease--free equilibrium $E_{0}=\left( \frac{\Lambda }{\mu }%
,0\right) $ are given, respectively, by%
\begin{equation*}
\begin{pmatrix}
\frac{\lambda \Lambda }{\mu }\varphi ^{\prime }(0) & 0 \\
0 & 0%
\end{pmatrix}%
:=%
\begin{pmatrix}
S & 0 \\
0 & 0%
\end{pmatrix}%
,
\end{equation*}%
and%
\begin{equation*}
\begin{pmatrix}
\sigma & 0 \\
\frac{\lambda \Lambda }{\mu }\varphi ^{\prime }(0) & \mu%
\end{pmatrix}%
:=%
\begin{pmatrix}
V & 0 \\
V_{1} & V_{2}%
\end{pmatrix}%
.
\end{equation*}%
The basic reproduction number $R_{0}$ is simply the spectral radius of the
next generation matrix%
\begin{equation*}
K=SV^{-1}=\left( \frac{\lambda \Lambda }{\mu }\varphi ^{\prime }(0)\right)
\left( \sigma \right) ^{-1}=\frac{\lambda \Lambda }{\mu \sigma }\varphi
^{\prime }(0),
\end{equation*}%
which is given by%
\begin{equation*}
R_{0}=\rho \left( SV^{-1}\right) =\frac{\Lambda \lambda }{\mu \sigma }%
\varphi ^{\prime }(0).
\end{equation*}

\subsection{The Local ODE Stability}

Now that we have identified two constant steady states for system (\ref{e1.6}%
)--(\ref{e1.7}), namely $E_{0}=\left( \frac{\Lambda }{\mu },0\right) $ and $%
E^{\ast }=\left( u^{\ast },v^{\ast }\right) $, we move to study their local
asymptotic stability as described in the following theorem.

\begin{theorem}
For system (\ref{e1.6})--(\ref{e1.7}):

\begin{itemize}
\item[(i)] If $R_{0}\leq 1$, the disease-free equilibrium solution$\ E_{0}$ $%
=\left( \frac{\Lambda }{\mu },0\right) $ is the only steady state of the
system and is locally asymptotically stable.

\item[(ii)] If $R_{0}>1$, $E_{0}$ is unstable and the positive constant
endemic steady state $E^{\ast }$ $=\left( u^{\ast },v^{\ast }\right) $ is
locally asymptotically stable.
\end{itemize}
\end{theorem}

\begin{proof}
To prove the local asymptotic stability of the constant steady states, we
make advantage of the Jacobian matrix, which may be given by%
\begin{equation*}
J\left( u,v\right) =\left(
\begin{tabular}{ll}
$F_{u}\left( u,v\right) $ & $F_{v}\left( u,v\right) $ \\
$G_{u}\left( u,v\right) $ & $G_{v}\left( u,v\right) $%
\end{tabular}%
\right) ,
\end{equation*}%
where $F_{u}$, $F_{v}$, $G_{u}$, and $G_{v}$ are the first partial
derivatives. Evaluating the derivatives yields%
\begin{equation}
J\left( u,v\right) =\left(
\begin{tabular}{ll}
$-\lambda \varphi \left( v\right) -\mu $ & $-\lambda u\varphi ^{\prime
}\left( v\right) $ \\
$\lambda \varphi \left( v\right) $ & $\lambda u\varphi ^{\prime }\left(
v\right) -\sigma $%
\end{tabular}%
\right) .  \label{e1.11}
\end{equation}%
Evaluating $J\left( u,v\right) $ at $E_{0}$ $=\left( \frac{\Lambda }{\mu }%
,0\right) $ with (\ref{e1.4}) in mind yields
\begin{equation*}
J\left( E_{0}\right) =\left(
\begin{tabular}{ll}
$-\mu $ & $-\lambda \frac{\Lambda }{\mu }\varphi ^{\prime }\left( 0\right) $
\\
$0$ & $\lambda \frac{\Lambda }{\mu }\varphi ^{\prime }\left( 0\right)
-\sigma $%
\end{tabular}%
\right) .
\end{equation*}%
The local stability of $E_{0}$ $=\left( \frac{\Lambda }{\mu },0\right) $
rests on the nature of the eigenvalues corresponding to $J\left(
E_{0}\right) $. The eigenvalues can be easily shown to be $\lambda _{1}=-\mu
<0$ and $\lambda _{2}=\lambda \frac{\Lambda }{\mu }\varphi ^{\prime }\left(
0\right) -\sigma $. It is easy to see that $\lambda _{2}<0$ if $R_{0}<1$,
leading to the asymptotic stability of $E_{0}$. Alternatively, if $R_{0}=1$,
the eigenvalues reduce to $\lambda _{1}=-\mu $ and $\lambda _{2}=0$, which
again means that $E_{0}$ is asymptotically stable.

The second case if where $R_{0}>1$. The equilibrium $E_{0}$ is clearly
unstable but the system possesses a positive endemic equilibrium $E^{\ast }$
$=\left( u^{\ast },v^{\ast }\right) $ where $u^{\ast },v^{\ast }>0$. Since $%
E^{\ast }$ satisfies (\ref{e1.8}), we have
\begin{equation}
\left \{
\begin{array}{l}
\Lambda =\lambda u^{\ast }\varphi \left( v^{\ast }\right) +\mu u^{\ast }, \\
\sigma =\frac{\lambda u^{\ast }\varphi \left( v^{\ast }\right) }{v^{\ast }}.%
\end{array}%
\right.  \label{e1.12}
\end{equation}%
Evaluating the Jacobian matrix (\ref{e1.11}) at $E^{\ast }$ $=\left( u^{\ast
},v^{\ast }\right) $ yields%
\begin{equation*}
J\left( u^{\ast },v^{\ast }\right) =\left(
\begin{tabular}{ll}
$-\lambda \varphi \left( v^{\ast }\right) -\mu $ & $-\lambda u^{\ast
}\varphi ^{\prime }\left( v^{\ast }\right) $ \\
$\lambda \varphi \left( v^{\ast }\right) $ & $\lambda u^{\ast }\varphi
^{\prime }\left( v^{\ast }\right) -\sigma $%
\end{tabular}%
\right) ,
\end{equation*}%
which has trace%
\begin{equation}
tr(J( u^{\ast },v^{\ast }) )= - ( \lambda \varphi ( v^{\ast }) +\mu ) -\sigma +\lambda u^{\ast }\varphi ^{\prime
}\left( v^{\ast }\right) .
\end{equation}%
Using (\ref{e1.12}) and (\ref{e1.5}), this can be rewritten as%
\begin{equation*}
tr(J\left( u^{\ast },v^{\ast }\right) )=-\frac{\Lambda }{u^{\ast }}-\lambda
u^{\ast }\left[ \frac{\varphi \left( v^{\ast }\right) }{v^{\ast }}-\varphi
^{\prime }\left( v^{\ast }\right) \right] <0.
\end{equation*}%
The determinant of the Jacobian may be given by%
\begin{equation*}
\det J\left( u^{\ast },v^{\ast }\right) =\lambda \sigma \varphi \left(
v^{\ast }\right) +\mu \sigma -\mu \lambda u^{\ast }\varphi ^{\prime }\left(
v^{\ast }\right) .
\end{equation*}%
Using (\ref{e1.12}), we obtain%
\begin{eqnarray*}
\det \left( J\left( u^{\ast },v^{\ast }\right) \right) &=&\lambda \frac{%
\lambda u^{\ast }\varphi \left( v^{\ast }\right) }{v^{\ast }}\varphi \left(
v^{\ast }\right) +\mu \frac{\lambda u^{\ast }\varphi \left( v^{\ast }\right)
}{v^{\ast }}-\mu \lambda u^{\ast }\varphi ^{\prime }\left( v^{\ast }\right)
\\
&=&\frac{\lambda ^{2}u^{\ast }\left( \varphi \left( v^{\ast }\right) \right)
^{2}}{v^{\ast }}+\mu \lambda u^{\ast }\left[ \frac{\varphi \left( v^{\ast
}\right) }{v^{\ast }}-\varphi ^{\prime }\left( v^{\ast }\right) \right] ,
\end{eqnarray*}%
and from (\ref{e1.5}), we obtain $\det \left( J\left( u^{\ast },v^{\ast
}\right) \right) >0$. Hence, the equilibrium $E^{\ast }$ $=\left( u^{\ast
},v^{\ast }\right) $ is locally asymptotically stable.
\end{proof}

\subsection{Global Existence of Solutions}

Assuming that the function $\varphi $ satisfies conditions (\ref{e1.4})--(%
\ref{e1.5}), the following proposition establishes that solutions of system (%
\ref{e1.1})--(\ref{e1.3}) exist globally in time and are bounded by a
parameter dependent constant. First, however, we state a lemma that was
developed in \cite{Bendoukha2017}\ and which will become useful later on.

\begin{lemma}
\label{Lemma1}Condition (\ref{e1.5}) implies%
\begin{equation}
0<\frac{\varphi \left( v\right) }{v}\leq \varphi ^{\prime }\left( 0\right)
\text{ for all }v>0.  \label{e1.13}
\end{equation}
\end{lemma}

\begin{proposition}
\label{proposition_exist}For any initial conditions $(u_{0},v_{0})\in C(%
\overline{\Omega })\times C(\overline{\Omega })$, the solution $%
(u(t,x),v(t,x))$ of system (\ref{e1.1})--(\ref{e1.3}) exists uniquely and
globally in time. In addition, there exists a positive constant $\mathcal{C}%
(u_{0},v_{0},\Lambda ,\mu ,\lambda ,\sigma )$ such that $\forall t>0$,%
\begin{equation}
\left \Vert u(t,.)\right \Vert _{L^{\infty }(\Omega )}+\left \Vert
v(t,.)\right \Vert _{L^{\infty }(\Omega )}\leq \mathcal{C}.  \label{ex_eq_1}
\end{equation}%
Furthermore, there exists a positive constant $\widetilde{\mathcal{C}}%
(\Lambda ,\mu ,\lambda ,\sigma )$ such that $\forall t\geq T$ for some large
$T>0$,%
\begin{equation}
\left \Vert u(t,.)\right \Vert _{L^{\infty }(\Omega )}+\left \Vert
v(t,.)\right \Vert _{L^{\infty }(\Omega )}\leq \widetilde{\mathcal{C}}.
\label{ex_eq_2}
\end{equation}
\end{proposition}

\begin{proof}
Let us define $\mathcal{X}=\mathcal{Y}\times \mathcal{Y}$ with $\mathcal{Y}%
=C(\overline{\Omega })$. Given%
\begin{equation*}
\psi =\left(
\begin{matrix}
\psi _{1} \\
\psi _{2}%
\end{matrix}%
\right) ,
\end{equation*}%
the norm on $\mathcal{X}$ is defined as $\left \Vert \psi \right \Vert
=\left \Vert \psi _{1}\right \Vert _{\mathcal{Y}}+\left \Vert \psi
_{2}\right \Vert _{\mathcal{Y}}$. Hence, $(\mathcal{X},\left \Vert
.\right \Vert )$ is a Banach space. Let us also define the unbounded operator
\begin{equation*}
\tilde{A}:D(\tilde{A})\subset \mathcal{X}\longrightarrow \mathcal{X}
\end{equation*}%
such that%
\begin{equation*}
\tilde{A}\left(
\begin{matrix}
u \\
v%
\end{matrix}%
\right) =(\tilde{A}_{1}u,\tilde{A}_{2}v)^{T}=(d_{1}\Delta ,d_{2}\Delta )^{T},
\end{equation*}%
and%
\begin{equation*}
D(\tilde{A})=D(\tilde{A}_{1})\times D(\tilde{A}_{2}),
\end{equation*}%
with%
\begin{equation*}
D(\tilde{A}_{i})=\left \{
\begin{matrix}
\varphi \in C^{2}(\Omega )\cap C^{1}(\overline{\Omega }): & \tilde{A}%
_{i}\varphi \in C(\overline{\Omega }), & \frac{\partial \varphi }{\partial
\nu }=0 & \text{on }\partial \Omega
\end{matrix}%
\right \} ,\  \  \text{for }i=1,2.
\end{equation*}%
We may write system (\ref{e1.1})--(\ref{e1.3}) in abstract form in $\mathcal{%
X}$ as%
\begin{equation*}
\left \{
\begin{array}{ll}
\frac{dU}{dt}(t)=AU(t)+\Upsilon (U(t)), & t>0, \\
U(0)=U_{0}\in \mathcal{X}, &
\end{array}%
\right.
\end{equation*}%
where $U(t)=(u(t,.),v(t,.))^{T}$, $\Upsilon (U(t))=(F(U(t)),G(U(t)))^{T}$
and $A= diag(A_{1},A_{2})$, in which $A_{i}$ is the closure of $\tilde{A}_{i}
$ in $\mathcal{Y}$ for $i=1,2$. The operator $A$ is the infinitesimal
generator of an analytical semigroup of bounded linear operators $\left \{
T(t)\right \} _{t\geq 0}$ on $\mathcal{X}$, \cite{Smith1995}. Since $\Upsilon
$ is locally Lipschitz in $\mathcal{X}$, it follows that system (\ref{e1.1}%
)--(\ref{e1.3}) admits a unique local solution $(u(t,x),v(t,x))$ for $t\in %
\left[ 0,T_{max}\right) $ and $x\in \overline{\Omega }$ where $T_{max}$ is
the maximal existence time, see \cite{Pazy1983}.

Let us now consider the case $u(t,x)\in (0,T_{max})\times \Omega $, which
can be formulated as%
\begin{equation}
\left \{
\begin{array}{ll}
\dfrac{\partial u}{\partial t}-d_{1}\Delta u=\Lambda -\mu u-\lambda u\varphi
\left( v\right) , & \text{in }(0,T_{max})\times \Omega , \\
u(0,x)=u_{0}(x), & \text{on }\Omega , \\
\frac{\partial u}{\partial \nu }=0, & \text{on }(0,T_{max})\times \partial
\Omega .%
\end{array}%
\right.  \label{eu}
\end{equation}%
We can easily observe that an upper solution exists for (\ref{eu}) for any
nonnegative function $v(t,x)$. This upper solution is given by%
\begin{equation*}
N_{1}:=\max \left \{ \frac{\Lambda }{\mu },\left \Vert u_{0}\right \Vert _{C(%
\overline{\Omega })}\right \} .
\end{equation*}%
Using the comparison principle from \cite{Protter1984}, we obtain $%
u(t,x)\leq N_{1}$ in $[0,T_{max})\times \overline{\Omega }$, and
consequently, $u$ is uniformly bounded in $[0,T_{max})\times \overline{%
\Omega }$. Integrating the equations of (\ref{e1.1}) over $\Omega $ and
taking the sum of the resulting two identities yields%
\begin{align}
\frac{d}{dt}\int_{\Omega }(u(t,x)+v(t,x))dx& =\left \vert \Omega \right
\vert \Lambda -\int_{\Omega }(\mu u(t,x)+\sigma v(t,x))dx,  \notag \\
& \leqslant \left \vert \Omega \right \vert \Lambda -\sigma _{0}\int_{\Omega
}(u(t,x)+v(t,x))dx,
\end{align}%
where $\sigma _{0}=\min \left \{ \mu ,\sigma \right \} $. From the
well--known Gronwall's inequality, we have for $t\in (0,T_{max})$,%
\begin{equation}
\int_{\Omega }(u(t,x)+v(t,x))dx\leq N_{2},
\end{equation}%
where $N_{2}>0$. Hence, for $t\in (0,T_{max})$,%
\begin{equation}
\int_{\Omega }v(t,x)dx\leq N_{2}.
\end{equation}%
Following the footsteps of \cite[Theorem 3.1]{Alikakos1979} and using the $v$%
--equation, we conclude $\exists N_{3}>0$ depending on $N_{2}$ such that $%
v(t,x)\leq N_{3}$ over $[0,T_{max})\times \overline{\Omega }$. Therefore, $v$
is also uniformly bounded in $[0,T_{max})\times \overline{\Omega }$. Using
the standard theory of semilinear parabolic systems described in \cite%
{Henry1993}, we deduce $T_{max}=\infty $.

When $T_{max}=\infty $, system (\ref{eu}) becomes%
\begin{equation}
\left \{
\begin{array}{ll}
\dfrac{\partial u}{\partial t}-d_{1}\Delta u=\Lambda -\mu u-\lambda u\varphi
\left( v\right) \leq \Lambda -\mu u, & \text{in }(0,\infty )\times \Omega ,
\\
u(0,x)=u_{0}(x)\leq \left \Vert u_{0}\right \Vert _{C(\overline{\Omega })},
& \text{on }\Omega , \\
\frac{\partial u}{\partial \nu }=0, & \text{on }(0,\infty )\times \partial
\Omega .%
\end{array}%
\right.  \label{eu_infty}
\end{equation}%
By means of the comparison principle we get $u(t,x)\leq \omega (t)$ for $%
t\in \lbrack 0,\infty )$, where $\omega (t)=\left \Vert u_{0}\right \Vert
_{C(\overline{\Omega })}e^{-\mu t}+\left( \frac{\Lambda }{\mu }\right)
(1-e^{-\mu t})$ is the unique solution of the problem%
\begin{equation}
\left \{
\begin{array}{ll}
\frac{d\omega }{dt}=\Lambda -\mu \omega , & t>0, \\
\omega (0)=\left \Vert u_{0}\right \Vert _{C(\overline{\Omega })}. &
\end{array}%
\right.  \label{eu_ode}
\end{equation}%
Consequently, for $x\in \overline{\Omega }$, we have%
\begin{equation*}
u(t,x)\leq \omega (t)\overset{t\rightarrow \infty }{\rightarrow }\frac{%
\Lambda }{\mu }.
\end{equation*}%
Therefore, we have an upper bound for $\left \Vert u(t,.)\right \Vert
_{L^{\infty }(\Omega )}$ independent of the initial conditions given a
sufficiently large $t$. By applying \cite[Lemma 3.1]{Peng2012} we find that $%
\left \Vert v(t,.)\right \Vert _{L^{\infty }(\Omega )}$ is also bounded by a
positive constant independent of the initial conditions for a large enough $%
t $.
\end{proof}

\begin{remark}
From Lemma \ref{Lemma1}, we obtain%
\begin{equation}
\begin{matrix}
\varphi (v)<\varphi ^{\prime }(0)e^{v}, & \forall v\in (0,\infty ).%
\end{matrix}%
\end{equation}%
Thus, we can establish the global existence of solutions for system (\ref%
{e1.1})--(\ref{e1.3}) and validate (\ref{ex_eq_1}) by following the work in\
\cite[Theorem 3.4]{Djebara2019}.
\end{remark}

\subsection{The Local PDE Stability}

We have already established sufficient conditions for the local asymptotic
stability of system (\ref{e1.5})--(\ref{e1.7}) in the ODE scenario. Let us
now examine the more general PDE case (\ref{e1.1})--(\ref{e1.3}).

\begin{theorem}
For system (\ref{e1.1})--(\ref{e1.3}):\newline
(i) If $R_{0}\leq 1$, the disease--free equilibrium$\ E_{0}$ $=\left( \frac{%
\Lambda }{\mu },0\right) $ is locally asymptotically stable.\newline
(ii) If $R_{0}>1$, the positive constant endemic steady equilibrium $E^{\ast
}$ $=\left( u^{\ast },v^{\ast }\right) $\ is locally asymptotically stable.
\end{theorem}

\begin{proof}
(i) In the presence of diffusion, the equilibrium point $E_{0}$ $=\left(
\frac{\Lambda }{\mu },0\right) $ satisfies%
\begin{equation*}
\left \{
\begin{array}{l}
d_{1}\Delta u+\Lambda -\lambda u^{\ast }\varphi \left( v^{\ast }\right) -\mu
u^{\ast }=0, \\
d_{2}\Delta v+\lambda u^{\ast }\varphi \left( v^{\ast }\right) -\sigma
v^{\ast }=0,%
\end{array}%
\right.
\end{equation*}%
with Neumann boundaries%
\begin{equation*}
\dfrac{\partial u}{\partial \nu }=\dfrac{\partial v}{\partial \nu }=0\  \text{%
on }\mathbb{R}^{+}\times \partial \Omega .
\end{equation*}%
We denote the indefinite sequence of positive eigenvalues for the Laplacian
operator $\Delta $ over $\Omega $ with Neumann boundary conditions by $%
0=\lambda _{0}<\lambda _{1}\leq \lambda _{2}\leq \lambda _{3}\leq
...\nearrow ^{+\infty }$. Note that the first eigenfunction is a constant,
which is why the corresponding eigenvalue is equal to zero. The
corresponding sequence of eigenfunctions is denoted by $\left( \Phi
_{ij}\right) _{j=\overline{1,m}_{i}}$, where $m_{i}\geq 1$ $\mathtt{is}$ the
algebraic multiplicity of $\lambda _{i}$. These functions are the solutions
of%
\begin{equation*}
\left \{
\begin{array}{l}
-\Delta \Phi _{ij}=\lambda _{i}\Phi _{ij}\text{ \ in }\Omega , \\
\dfrac{\partial \Phi _{ij}}{\partial \nu }=0\text{ \ on }\partial \Omega .%
\end{array}%
\right.
\end{equation*}%
The eigenfunctions are normalized according to%
\begin{equation*}
\left \Vert \Phi _{ij}\right \Vert _{2}=\int_{\Omega }\Phi _{ij}^{2}\left(
x\right) dx=1.
\end{equation*}%
The set of eigenfunctions $\left \{ \Phi _{ij}:i\geq 0,j=\overline{1,m}%
_{i}\right \} $ forms a complete orthonormal basis in $L^{2}\left( \Omega
\right) $. In order to establish the local asymptotic stability of the
steady states, we must examine all the eigenvalues of the linearizing
operator and if they all have negative real parts, then the solution is
locally asymptotically stable. The linearizing operator may be given by%
\begin{equation*}
\mathcal{L}\left( E_{0}\right) \mathcal{=}\left(
\begin{tabular}{ll}
$d_{1}\Delta -\mu $ & $-\lambda \frac{\Lambda }{\mu }\varphi ^{\prime
}\left( 0\right) $ \\
$0$ & $d_{2}\Delta +\lambda \frac{\Lambda }{\mu }\varphi ^{\prime }\left(
0\right) -\sigma $%
\end{tabular}%
\right) .
\end{equation*}%
Using the same method from \cite{Abdelmalek2016}, the stability of $E_{0}$
reduces to examining the eigenvalues of the matrices%
\begin{equation*}
J_{i}\left( E_{0}\right) \mathcal{=}\left(
\begin{tabular}{ll}
$-d_{1}\lambda _{i}-\mu $ & $-\lambda \frac{\Lambda }{\mu }\varphi ^{\prime
}\left( 0\right) $ \\
$0$ & $-d_{2}\lambda _{i}+\lambda \frac{\Lambda }{\mu }\varphi ^{\prime
}\left( 0\right) -\sigma $%
\end{tabular}%
\right) \text{ for all }i\geq 0,
\end{equation*}%
which are given for all $i\geq 0$ by%
\begin{equation*}
\left \{
\begin{array}{l}
r_{i1}=-d_{1}\lambda _{i}-\mu , \\
r_{i2}=-d_{2}\lambda _{i}+\lambda \frac{\Lambda }{\mu }\varphi ^{\prime
}\left( 0\right) -\sigma .%
\end{array}%
\right.
\end{equation*}%
Since the Laplacian eigenvalues are positive and in ascending order, both $%
r_{i1}$ and $r_{i2}$ clearly\ have negative real parts for $R_{0}\leq 1$
leading to the local stability of $E_{0}$ $=\left( \frac{\Lambda }{\mu }%
,0\right) $.

(ii) The second steady state $E^{\ast }=\left( u^{\ast },v^{\ast }\right) $
satisfies%
\begin{equation*}
\left \{
\begin{array}{l}
d_{1}\Delta u+\Lambda -\lambda u^{\ast }\varphi \left( v^{\ast }\right) -\mu
u^{\ast }=0, \\
d_{2}\Delta v+\lambda u^{\ast }\varphi \left( v^{\ast }\right) -\sigma
v^{\ast }=0, \\
\dfrac{\partial u}{\partial \nu }=\dfrac{\partial v}{\partial \nu }=0\  \text{%
on }\mathbb{R}^{+}\times \partial \Omega .%
\end{array}%
\right.
\end{equation*}

The corresponding linearization operator is%
\begin{equation*}
\mathcal{L}\left( E^{\ast }\right) =\left(
\begin{tabular}{ll}
$d_{1}\Delta -\lambda \varphi \left( v^{\ast }\right) -\mu $ & $-\lambda
u^{\ast }\varphi ^{\prime }\left( v^{\ast }\right) $ \\
$\lambda \varphi \left( v^{\ast }\right) $ & $d_{2}\Delta +\lambda u^{\ast
}\varphi ^{\prime }\left( v^{\ast }\right) -\sigma $%
\end{tabular}%
\right) .
\end{equation*}%
Hence, the stability of $E^{\ast }$ rests on the negativity of the real
parts of the eigenvalues of matrices%
\begin{equation*}
J_{i}\left( E^{\ast }\right) =\left(
\begin{tabular}{ll}
$-d_{1}\lambda _{i}-\lambda \varphi \left( v^{\ast }\right) -\mu $ & $%
-\lambda u^{\ast }\varphi ^{\prime }\left( v^{\ast }\right) $ \\
$\lambda \varphi \left( v^{\ast }\right) $ & $-d_{2}\lambda _{i}+\lambda
u^{\ast }\varphi ^{\prime }\left( v^{\ast }\right) -\sigma $%
\end{tabular}%
\right) ,\text{ for all }i\geq 0,
\end{equation*}%
which is guaranteed if the trace and determinant of $J_{i}\left( E^{\ast
}\right) $ satisfies the conditions $ tr(J_{i}\left( E^{\ast }\right) )<0$
and $\det \left( J_{i}\left( E^{\ast }\right) \right) >0,$ for all $i\geq 0$%
. The trace of $J_{i}\left( E^{\ast }\right) $ is given by%
\begin{eqnarray*}
tr(J_{i}\left( E^{\ast }\right) ) &=&-d_{1}\lambda _{i}-d_{2}\lambda
_{i}-\lambda \varphi \left( v^{\ast }\right) -\mu +\lambda u^{\ast }\varphi
^{\prime }\left( v^{\ast }\right) -\sigma , \\
&=&-\lambda _{i}\left( d_{1}+d_{2}\right) +tra(J\left( u^{\ast },v^{\ast
}\right) ).
\end{eqnarray*}%
Since $tr(J\left( u^{\ast },v^{\ast }\right) )<0$, it follows that $tr%
(J_{i}\left( E^{\ast }\right) )<0$ for all $i\geq 0$. The determinant is
given by%
\begin{eqnarray*}
\det \left( J_{i}\left( E^{\ast }\right) \right) &=&d_{1}d_{2}\lambda
_{i}^{2}+\lambda _{i}\left[ -d_{1}\lambda u^{\ast }\varphi ^{\prime }\left(
v^{\ast }\right) +d_{1}\sigma +\lambda \varphi \left( v^{\ast }\right)
d_{2}+\mu d_{2}\right] \\
&&+\lambda \sigma \varphi \left( v^{\ast }\right) +\mu \sigma -\mu \lambda
u^{\ast }\varphi ^{\prime }\left( v^{\ast }\right) , \\
&=&d_{1}d_{2}\lambda _{i}^{2}+\lambda _{i}H_{0}+\det \left( J\left( u^{\ast
},v^{\ast }\right) \right) \text{ for all }i\geq 0,
\end{eqnarray*}%
where%
\begin{equation*}
H_{0}=-d_{1}\lambda u^{\ast }\varphi ^{\prime }\left( v^{\ast }\right)
+d_{1}\sigma +\lambda \varphi \left( v^{\ast }\right) d_{2}+\mu d_{2}.
\end{equation*}%
Using$\ $(\ref{e1.5}) and (\ref{e1.12}), it holds that%
\begin{eqnarray*}
H_{0} &\geq &-d_{1}\lambda u^{\ast }\frac{\varphi \left( v^{\ast }\right) }{%
v^{\ast }}+d_{1}\frac{\lambda u^{\ast }\varphi \left( v^{\ast }\right) }{%
v^{\ast }}+\lambda \varphi \left( v^{\ast }\right) d_{2}+\mu d_{2}, \\
&=&d_{2}\left( \lambda \varphi \left( v^{\ast }\right) +\mu \right) , \\
&=&d_{2}\frac{\Lambda }{u^{\ast }}>0,
\end{eqnarray*}%
which leads to%
\begin{equation*}
\det \left( J_{i}\left( E^{\ast }\right) \right) =d_{1}d_{2}\lambda
_{i}^{2}+\lambda _{i}H_{0}+\det \left( J\left( u^{\ast },v^{\ast }\right)
\right) >0\text{ for all }i\geq 0.
\end{equation*}%
Hence, $E^{\ast }$ is locally asymptotically stable.
\end{proof}

\section{Global asymptotic stability\label{SecGlobal}}

Next, we study the global asymptotic stability of the two steady states $%
E_{0}$ and $E^{\ast }$. The global stability depends on the reproduction
number $R_{0}$, which is why we have decided to treat the scenarios $%
R_{0}\leq 1$\ and $R_{0}>1$\ separately. First, however, let us state a
necessary lemma taken from \cite{Sigdel2014} that will aid with the proofs
to come.

\begin{lemma}
\label{Lemma 2}Given that\ $\varphi $ satisfies criterion (\ref{e1.5}) and%
\begin{equation}
L\left( x\right) =x-1-\ln \left( x\right) \text{, for all }x>0,
\label{e1.18}
\end{equation}%
the inequality%
\begin{equation}
L\left( \frac{\varphi (v)}{\varphi (v^{\ast })}\right) \leq L\left( \frac{v}{%
v^{\ast }}\right) ,  \label{e1.14}
\end{equation}%
where $v^{\ast }$ is the second component of the equilibrium point $E^{\ast
} $, holds.
\end{lemma}

\begin{proof}
Condition (\ref{e1.5}) guarantees that $\frac{\varphi \left( v\right) }{v}$
is a decreasing function for all $v>0$. We may separate the proof into two
regions:

\begin{itemize}
\item[1.] The first region is $v\geq v^{\ast }$. Since $\frac{\varphi \left(
v\right) }{v}$ is a decreasing function, we have%
\begin{equation*}
\frac{\varphi \left( v\right) }{\varphi \left( v^{\ast }\right) }\leq \frac{v%
}{v^{\ast }}.
\end{equation*}%
Note that (\ref{e1.5}) implies that $\varphi $ is non decreasing, which
leads to%
\begin{equation*}
\varphi \left( v\right) \geq \varphi \left( v^{\ast }\right) ,
\end{equation*}%
and, consequently,%
\begin{equation*}
1\leq \frac{\varphi \left( v\right) }{\varphi \left( v^{\ast }\right) }\leq
\frac{v}{v^{\ast }}.
\end{equation*}%
Since $L$ is increasing for all $x>1$, (\ref{e1.14}) follows.%
\begin{equation*}
L\left( \frac{\varphi (v)}{\varphi (v^{\ast })}\right) \leq L\left( \frac{v}{%
v^{\ast }}\right)
\end{equation*}
for all $v\geq v^{\ast }$.

\item[2.] The second region is $0<v<v^{\ast }$. Again, since $\frac{\varphi
\left( v\right) }{v}$ is a decreasing function, we have%
\begin{equation*}
\frac{\varphi \left( v\right) }{\varphi \left( v^{\ast }\right) }>\frac{v}{%
v^{\ast }},
\end{equation*}%
and given the non--decreasing nature of $\varphi $, we end up with%
\begin{equation*}
\varphi \left( v\right) <\varphi \left( v^{\ast }\right) .
\end{equation*}%
As a result, we get%
\begin{equation*}
1>\frac{\varphi \left( v\right) }{\varphi \left( v^{\ast }\right) }>\frac{v}{%
v^{\ast }}>0.
\end{equation*}%
Since $L$ is decreasing for $0<x<1$, (\ref{e1.14}) holds.
\end{itemize}
\end{proof}

\subsection{Global Asymptotic Stability with $R_{0}\leq 1$}

To establish the global asymptotic stability of the disease free equilibrium
$E_{0}$, we consider%
\begin{equation*}
\mathcal{V}_{\theta }\left( t\right) =\int_{\Omega }\left[ uv+\frac{\theta }{%
2}\left( u-\tfrac{\Lambda }{\mu }\right) ^{2}+\frac{1}{2}v^{2}+\tfrac{%
\Lambda }{\sigma }v\right] dx,\text{ where }\theta >0,
\end{equation*}%
as our candidate Lyapunov function. The aim of this subsection is to obtain
a weaker condition than that of \cite{Djebara2019}.

\begin{theorem}
\label{TheoG1}If $R_{0}\leq 1$, $E_{0}$ is a globally asymptotically stable
disease--free steady state for system (\ref{e1.1})--(\ref{e1.3}) under the
assumption%
\begin{equation}
\varphi ^{\prime }\left( 0\right) \  \leq \frac{\mu +\sigma }{\lambda \left(
\theta \frac{\Lambda }{\mu }+\frac{\Lambda }{\sigma }\right) },
\label{e1.15}
\end{equation}%
with%
\begin{equation}
\theta \geq \frac{\left( d_{1}+d_{2}\right) ^{2}}{4d_{1}d_{2}}.
\label{e1.16}
\end{equation}
\end{theorem}

\begin{proof}
To prove that $E_{0}=\left( \frac{\Lambda }{\mu },0\right) $ is globally
asymptotically stable, we must show that $\mathcal{V}_{\theta }\left(
t\right) $ is a Lyapunov function. The positive definiteness of $\mathcal{V}%
_{\theta }\left( t\right) $ is evident. Evaluating its derivative with
respect to time gives%
\begin{equation*}
\frac{d}{dt}\mathcal{V}_{\theta }\left( t\right) =\int_{\Omega }\left(
\tfrac{\partial u}{\partial t}v+u\tfrac{\partial v}{\partial t}\right)
dx+\theta \int_{\Omega }\left( u-\tfrac{\Lambda }{\mu }\right) \tfrac{%
\partial u}{\partial t}dx+\int_{\Omega }v\tfrac{\partial v}{\partial t}dx+%
\tfrac{\Lambda }{\sigma }\int_{\Omega }\tfrac{\partial v}{\partial t}dx.
\end{equation*}%
Substituting the partial derivatives $\tfrac{\partial u}{\partial t}$ and $%
\tfrac{\partial v}{\partial t}$ with their respective values from (\ref{e1.1}%
) and applying Green's formula with the assumed Neumann boundaries leads to%
\begin{equation*}
\frac{d}{dt}\mathcal{V}_{\theta }\left( t\right) =I_{1}+I_{2}+I_{3},
\end{equation*}%
where%
\begin{equation*}
I_{1}=-\left( d_{1}+d_{2}\right) \int_{\Omega }\nabla u\nabla vdx+\Lambda
\int_{\Omega }vdx-\lambda \int_{\Omega }uv\varphi \left( v\right) dx+\lambda
\int_{\Omega }u^{2}\varphi \left( v\right) dx-\left( \sigma +\mu \right)
\int_{\Omega }uvdx,
\end{equation*}%
\begin{equation*}
I_{2}=-d_{1}\theta \int_{\Omega }\left \vert \nabla u\right \vert ^{2}dx-\mu
\theta \int_{\Omega }\left( u-\tfrac{\Lambda }{\mu }\right) ^{2}dx-\lambda
\theta \int_{\Omega }u^{2}\varphi \left( v\right) dx+\lambda \theta \frac{%
\Lambda }{\mu }\int_{\Omega }u\varphi \left( v\right) dx,
\end{equation*}%
and%
\begin{equation*}
I_{3}=-d_{2}\int_{\Omega }\left \vert \nabla v\right \vert ^{2}dx+\lambda
\int_{\Omega }uv\varphi \left( v\right) dx-\sigma \int_{\Omega
}v^{2}dx+\lambda \tfrac{\Lambda }{\sigma }\int_{\Omega }u\varphi \left(
v\right) dx-\Lambda \int_{\Omega }vdx.
\end{equation*}%
Taking the sum of terms $I_{1}$, $I_{2}$, and $I_{3}$ yields%
\begin{eqnarray*}
\frac{d}{dt}\mathcal{V}_{\theta }\left( t\right) &=&-d_{1}\theta
\int_{\Omega }\left \vert \nabla u\right \vert ^{2}dx-\left(
d_{1}+d_{2}\right) \int_{\Omega }\nabla u\nabla vdx-d_{2}\int_{\Omega }\left
\vert \nabla v\right \vert ^{2}dx \\
&&+\lambda \left( 1-\theta \right) \int_{\Omega }u^{2}\varphi \left(
v\right) dx-\left( \sigma +\mu \right) \int_{\Omega }uvdx-\mu \theta
\int_{\Omega }\left( u-\tfrac{\Lambda }{\mu }\right) ^{2}dx \\
&&-\sigma \int_{\Omega }v^{2}dx+\lambda \left( \theta \frac{\Lambda }{\mu }+%
\tfrac{\Lambda }{\sigma }\right) \int_{\Omega }u\varphi \left( v\right) dx.
\\
&=&I+J,
\end{eqnarray*}

where%
\begin{equation*}
I=-d_{1}\theta \int_{\Omega }\left \vert \nabla u\right \vert ^{2}dx-\left(
d_{1}+d_{2}\right) \int_{\Omega }\nabla u\nabla vdx-d_{2}\int_{\Omega }\left
\vert \nabla v\right \vert ^{2}dx,
\end{equation*}%
and%
\begin{eqnarray*}
J &=&\lambda \left( 1-\theta \right) \int_{\Omega }u^{2}\varphi \left(
v\right) dx+\lambda \left( \theta \frac{\Lambda }{\mu }+\tfrac{\Lambda }{%
\sigma }\right) \int_{\Omega }u\varphi \left( v\right) dx \\
&&\  \  \  \  \  \  \  \  \  \  \  \  \  \  \  \ -\left( \sigma +\mu \right) \int_{\Omega
}uvdx-\mu \theta \int_{\Omega }\left( u-\tfrac{\Lambda }{\mu }\right)
^{2}dx-\sigma \int_{\Omega }v^{2}dx.
\end{eqnarray*}%
The first term may be written as%
\begin{equation*}
I=-\int_{\Omega }\emph{Q}\left( \nabla u,\nabla v\right) dx,
\end{equation*}%
where $\emph{Q}$ is a quadratic form with respect to $\nabla u$ and $\nabla
v $, i.e.%
\begin{equation*}
\emph{Q}\left( \nabla u,\nabla v\right) =d_{1}\theta \left \vert \nabla
u\right \vert ^{2}+\left( d_{1}+d_{2}\right) \nabla u\nabla v+d_{2}\left
\vert \nabla v\right \vert ^{2}.
\end{equation*}%
It is easy to see that $\emph{Q}\left( \nabla u,\nabla v\right) $ is
non--negative as $\theta $, $d_{1}$, and $d_{2}$ satisfy the conditions $%
d_{1}\theta >0$ and $\  \theta \geq \frac{\left( d_{1}+d_{2}\right) ^{2}}{%
4d_{1}d_{2}}$, from which we obtain $I\leq 0$.

On the other hand, using the inequality$\  \theta \geq \frac{\left(
d_{1}+d_{2}\right) ^{2}}{4d_{1}d_{2}}\geq 1$, we have%
\begin{equation*}
J\leq \lambda \left( \theta \frac{\Lambda }{\mu }+\tfrac{\Lambda }{\sigma }%
\right) \int_{\Omega }u\varphi \left( v\right) dx-\left( \sigma +\mu \right)
\int_{\Omega }uvdx-\mu \theta \int_{\Omega }\left( u-\tfrac{\Lambda }{\mu }%
\right) ^{2}dx-\sigma \int_{\Omega }v^{2}dx.
\end{equation*}%
Applying Lemma~\ref{Lemma1} yields%
\begin{equation}
J\leq \int_{\Omega }\left[ \lambda \left( \theta \tfrac{\Lambda }{\mu }+%
\tfrac{\Lambda }{\sigma }\right) \varphi ^{\prime }\left( 0\right) -\left(
\mu +\sigma \right) \right] uvdx-\mu \theta \int_{\Omega }\left( u-\tfrac{%
\Lambda }{\mu }\right) ^{2}dx-\sigma \int_{\Omega }v^{2}dx.
\end{equation}%
Assuming (\ref{e1.15}) holds, the derivative $\frac{d}{dt}\mathcal{V}%
_{\theta }\left( t\right) \leq 0$ for all $t\geq 0$ with $\mathcal{V}%
_{\theta }\left( t\right) =0$ being true when $\left( u,v\right) =\left(
\frac{\Lambda }{\mu },0\right) $. Finally, by Lyapunov's direct method, $%
E_{0}$ is globally asymptotically stable.
\end{proof}

\subsection{Global Asymptotic Stability with $R_{0}>1$}

\begin{theorem}
\label{TheoG2}If $R_{0}>1$, $E^{\ast }$ is a globally asymptotically stable
endemic steady--state for system (\ref{e1.1})--(\ref{e1.3}).
\end{theorem}

\begin{proof}
For this proof, we consider the candidate Lyapunov function
\begin{equation}
\mathcal{V}\left( t\right) =\int_{\Omega }\left[ u^{\ast }L\left( \frac{u}{%
u^{\ast }}\right) +v^{\ast }L\left( \frac{v}{v^{\ast }}\right) \right] dx,
\label{e1.19}
\end{equation}%
which is a positive definite and continuously differentiable function.
First, note that%
\begin{equation*}
\frac{d}{dt}L\left( \frac{u}{u^{\ast }}\right) =\frac{1}{u^{\ast }}\left( 1-%
\frac{u^{\ast }}{u}\right) \frac{du}{dt}.
\end{equation*}%
Differentiating $\mathcal{V}\left( t\right) $ with respect to time yields%
\begin{equation*}
\frac{d}{dt}\mathcal{V}\left( t\right) =\int_{\Omega }\left( 1-\frac{u^{\ast
}}{u}\right) \left( d_{1}\Delta u+\Lambda -\lambda u\varphi \left( v\right)
-\mu u\right) dx+\int_{\Omega }\left( 1-\frac{v^{\ast }}{v}\right) \left(
d_{2}\Delta v+\lambda u\varphi \left( v\right) -\sigma v\right) dx.
\end{equation*}%
Similar to the previous scenario, we apply Green's formula with Neumann
boundaries to expand the derivative to%
\begin{eqnarray*}
\frac{d}{dt}\mathcal{V}\left( t\right) &=&-d_{1}\int_{\Omega }\nabla \left(
1-\frac{u^{\ast }}{u}\right) \nabla udx+\Lambda \int_{\Omega }\left( 1-\frac{%
u^{\ast }}{u}\right) dx-\lambda \int_{\Omega }\left( 1-\frac{u^{\ast }}{u}%
\right) u\varphi \left( v\right) dx \\
&&-\mu \int_{\Omega }\left( 1-\frac{u^{\ast }}{u}\right)
udx-d_{2}\int_{\Omega }\nabla \left( 1-\frac{v^{\ast }}{v}\right) \nabla
vdx+\lambda \int_{\Omega }\left( 1-\frac{v^{\ast }}{v}\right) u\varphi
\left( v\right) dx \\
&&-\sigma \int_{\Omega }\left( 1-\frac{v^{\ast }}{v}\right) vdx \\
&=&I+J,
\end{eqnarray*}%
where%
\begin{equation}
I=-d_{1}\int_{\Omega }\frac{u^{\ast }}{u^{2}}\left \vert \nabla u\right
\vert ^{2}dx-d_{2}\int_{\Omega }\frac{v^{\ast }}{v^{2}}\left \vert \nabla
v\right \vert ^{2}dx-\int_{\Omega }\left[ d_{1}\frac{u^{\ast }}{u^{2}}\left
\vert \nabla u\right \vert ^{2}+d_{2}\frac{v^{\ast }}{v^{2}}\left \vert
\nabla v\right \vert ^{2}\right] dx\leq 0,
\end{equation}%
and%
\begin{equation}
J=\int_{\Omega }\left( 1-\frac{u^{\ast }}{u}\right) \left[ \Lambda -\lambda
u\varphi \left( v\right) -\mu u\right] dx+\int_{\Omega }\left( 1-\frac{%
v^{\ast }}{v}\right) \left[ \lambda u\varphi \left( v\right) -\sigma v\right]
dx.
\end{equation}%
Using (\ref{e1.12}) and simplifying the resulting equation leads to
\begin{align}
J& =\int_{\Omega }\mu u^{\ast }\left( 1-\frac{u}{u^{\ast }}\right) \left( 1-%
\frac{u^{\ast }}{u}\right) dx+\lambda u^{\ast }\varphi \left( v^{\ast
}\right) \int_{\Omega }\left( \frac{u\varphi \left( v\right) }{u^{\ast
}\varphi \left( v^{\ast }\right) }-\frac{v}{v^{\ast }}\right) \left( 1-\frac{%
v^{\ast }}{v}\right) dx  \notag  \label{e1.22} \\
& \  \  \  \  \  \ +\lambda u^{\ast }\varphi \left( v^{\ast }\right) \int_{\Omega
}\left( 1-\frac{u\varphi \left( v\right) }{u^{\ast }\varphi \left( v^{\ast
}\right) }\right) \left( 1-\frac{u^{\ast }}{u}\right) dx  \notag \\
& =\int_{\Omega }\left[ \mu u^{\ast }J_{1}+\lambda u^{\ast }\varphi \left(
v^{\ast }\right) J_{2}\right] dx,
\end{align}%
where%
\begin{equation*}
J_{1}=\left( 1-\frac{u^{\ast }}{u}\right) \left( 1-\frac{u}{u^{\ast }}%
\right) ,
\end{equation*}%
and%
\begin{equation*}
J_{2}=\frac{\varphi \left( v\right) }{\varphi \left( v^{\ast }\right) }-%
\frac{v}{v^{\ast }}-\frac{u\varphi \left( v\right) v^{\ast }}{u^{\ast
}\varphi \left( v^{\ast }\right) v}-\frac{u^{\ast }}{u}+2.
\end{equation*}%
Since $J_{1}$ and $J_{2}$ can be rewritten in the forms%
\begin{equation}
J_{1}=-L\left( \frac{u}{u^{\ast }}\right) -L\left( \frac{u^{\ast }}{u}%
\right) ,  \label{e1.23}
\end{equation}%
and%
\begin{equation}
J_{2}=-L\left( \frac{u^{\ast }}{u}\right) +L\left( \frac{\varphi \left(
v\right) }{\varphi \left( v^{\ast }\right) }\right) -L\left( \frac{v}{%
v^{\ast }}\right) -L\left( \frac{u\varphi \left( v\right) v^{\ast }}{u^{\ast
}\varphi \left( v^{\ast }\right) v}\right) ,  \label{e1.24}
\end{equation}%
it follows that%
\begin{align*}
J& =-\mu u^{\ast }\int_{\Omega }\left[ L\left( \frac{u}{u^{\ast }}\right)
+L\left( \frac{u^{\ast }}{u}\right) \right] dx-\lambda u^{\ast }\varphi
\left( v^{\ast }\right) \int_{\Omega }\left[ L\left( \frac{u^{\ast }}{u}%
\right) +L\left( \frac{u\varphi \left( v\right) v^{\ast }}{u^{\ast }\varphi
\left( v^{\ast }\right) v}\right) \right] dx \\
& \  \  \  \ +\lambda u^{\ast }\varphi \left( v^{\ast }\right) \int_{\Omega }
\left[ L\left( \frac{\varphi \left( v\right) }{\varphi \left( v^{\ast
}\right) }\right) -L\left( \frac{v}{v^{\ast }}\right) \right] dx.
\end{align*}%
Taking advantage of the established positivity of $L$ and applying Lemma \ref%
{Lemma 2}, we end up with%
\begin{equation*}
J\leq 0.
\end{equation*}%
Hence, $\frac{d}{dt}\mathcal{V}\left( t\right) \leq 0$ and, consequently, $%
\mathcal{V}$ is non--increasing in time with $\mathcal{V}\left( t\right) =0$
only at the steady state $E^{\ast }$. The global asymptotic stability of $%
E^{\ast }$ follows from Lyapunov's direct method's.
\end{proof}

\subsection{Important Remarks}

It is important to note that, for system (\ref{e1.1})--(\ref{e1.2}) in ODE
case, we have the following results:

\begin{itemize}
\item[(i)] For $R_{0}\leq 1$, the disease--free equilibrium $E_{0}$ is
globally asymptotically stable. This easy to establish as $d_{1}=d_{2}=0$
and, thus, it suffices to choose $\theta =1$ in $\mathcal{V}_{\theta }\left(
t\right) $.

\item[(ii)] For $R_{0}>1$ and provided that condition (\ref{e1.5}) is
satisfied, the endemic equilibrium $E^{\ast }$ $=\left( u^{\ast },v^{\ast
}\right) $ is globally asymptotically stable.
\end{itemize}

\section{Applications and Numerical Examples\label{SecExampl}}

Although the analytical methods we used to obtain the main results of this
study are theoretically sound, it is always beneficial to use numerical
solutions to illustrate and confirm their validity. We use consider an
incidence function of he form $u\varphi \left( v\right) $ satisfying (\ref%
{e1.5}) and assess the local and global asymptotic stability of the
disease--free equilibrium $E_{0}$\ when $R_{0}\leq 1$ and the endemic
equilibrium $E^{\ast }$ when $R_{0}>1$. Focus will be on the main results of
the paper as reported in Theorems \ref{TheoG1} (condition (\ref{e1.15})) and %
\ref{TheoG2}. In the following, we examine three different numerical
examples.

\subsection{First Example}

In this first example, we consider the function%
\begin{equation*}
\varphi \left( v\right) =\alpha v,\text{ for all }\alpha >0.
\end{equation*}%
The resulting problem is given by%
\begin{equation}
\left \{
\begin{array}{ll}
\dfrac{\partial u}{\partial t}-d_{1}\Delta u=-\lambda \alpha uv+\Lambda -\mu
u & \text{in }(0,\infty )\times \Omega , \\
\dfrac{\partial v}{\partial t}-d_{2}\Delta v=\lambda \alpha uv-\sigma v &
\text{in }(0,\infty )\times \Omega , \\
u(0,x)=u_{0}(x){,\  \ }v(0,x)=v_{0}(x) & \text{on }\Omega , \\
\frac{\partial u}{\partial \nu }=\frac{\partial v}{\partial \nu }=0, & \text{%
on }(0,\infty )\times \partial \Omega ,%
\end{array}%
\right.   \label{e1.25}
\end{equation}%
which is a special case of system (\ref{e1.1})--(\ref{e1.3}). In fact,
system (\ref{e1.25}) is identical to the bird system proposed in Section 3
of \cite{kim2010} and in \cite{Chinviriyasit2010} but with $d_{1}=d_{2}$.
The ODE scenario of this model was treated earlier in \cite{Korobeinikov2002}
and \cite{Zhou2003}, respectively. Conditions (\ref{e1.4}) and (\ref{e1.5})
are clearly satisfied as%
\begin{equation*}
\left \{
\begin{array}{l}
\varphi \left( 0\right) =0, \\
\varphi ^{\prime }\left( v\right) =\alpha >0, \\
\varphi ^{\prime }\left( 0\right) =\alpha , \\
\alpha v=v\varphi ^{\prime }\left( v\right) \leq \varphi \left( v\right)
=\alpha v.%
\end{array}%
\right.
\end{equation*}%
System (\ref{e1.25}) possesses two constant steady states
\begin{equation*}
E_{0}=\left( \frac{\Lambda }{\mu },0\right) \  \  \text{and}\  \ E^{\ast
}=\left( \frac{\sigma }{\lambda \alpha },\mu \sigma \left( R_{0}-1\right)
\right) .
\end{equation*}%
Note that the second steady state $E^{\ast }$ exists only when the
reproduction number $R_{0}=\frac{\lambda \Lambda }{\mu \sigma }\varphi
^{\prime }\left( 0\right) =\frac{\lambda \Lambda }{\mu \sigma }\alpha >1$
and is globally asymptotically stable. Also, note that the first steady
state $E_{0}$ is globally asymptotically stable unconditionally in the ODE
case and subject to
\begin{equation*}
\frac{(d_{1}+d_{2})^{2}}{4d_{1}d_{2}}\leqslant \theta \leqslant \frac{\mu }{%
\Lambda }\left( \frac{\mu +\sigma }{\lambda \alpha }-\frac{\Lambda }{\sigma }%
\right)
\end{equation*}%
when $d_{1}\neq d_{2}$, and to $\theta =1$ when $d_{1}=d_{2}$ in the PDE
case. As detailed in Table \ref{Tab1}, we use different sets of parameters
to obtain numerical solutions in the ODE and PDE. Note that throughout the
PDE simulations, we assume a single spatial dimension with $\Omega =(0,10)$.

\begin{table}[tbp]
\caption{Simulation parameters for example 1: system (\protect \ref{e1.25}).}
\label{Tab1}\centering%
\begin{tabular}{|l|l|l|l|l|l|l|l|l|l|l|l|l|}
\hline
& Set & Figure & $u_{0}$ & $v_{0}$ & $d_{1}$ & $d_{2}$ & $\lambda $ & $%
\alpha $ & $\sigma $ & $\Lambda $ & $\mu $ & $R_{0}$ \\ \hline
ODE & Set 1 & \ref{Fig1_ode}(a) & $6$ & $\frac{3}{2}$ & $-$ & $-$ & $\frac{1%
}{3}$ & $3$ & $2$ & $8$ & $1$ & $4$ \\ \cline{2-13}
case & Set 2 & \ref{Fig1_ode}(b) & $6$ & $\frac{3}{2}$ & $-$ & $-$ & $2$ & $%
\frac{1}{3}$ & $\frac{3}{2}$ & $6$ & $4$ & $0.8333$ \\ \hline
PDE & Set 1 & \ref{Fig1} & $4+\frac{\cos (x)}{10}$ & $5+\frac{\sin (x)}{10}$
& $3$ & $\frac{5}{4}$ & $\frac{1}{3}$ & $3$ & $2$ & $8$ & $1$ & $4$ \\
\cline{2-13}
case & Set 2 & \ref{Fig2} & $4+\frac{\cos (x)}{10}$ & $5+\frac{\sin (x)}{10}$
& $3$ & $\frac{5}{4}$ & $2$ & $\frac{1}{3}$ & $\frac{3}{2}$ & $6$ & $4$ & $%
0.8333$ \\ \hline
\end{tabular}%
\end{table}

The following is a description of the results:

\begin{itemize}
\item Figure \ref{Fig1_ode} shows the solutions in the ODE case subject to
sets 1 and 2, with $R_{0}=4$ and $R_{0}=0.8333$, respectively. In the first
case, as $R_{0}>1$, $E^{\ast }=(2,3)$ is globally asymptotically stable. In
the second case, $R_{0}\leq 1$ and $E_{0}=(\frac{3}{2},0)$ is globally
asymptotically stable.

\item Figure \ref{Fig1} depicts the solution in the PDE case subject to
parameter set 1, where $R_{0}=4>1$, which by Theorem \ref{TheoG2} means that
$E^{\ast }=(2,3)$ is globally asymptotically stable.

\item Figure \ref{Fig2} depicts the solution in the PDE case subject to
parameter set 2, where $R_{0}=0.8333\leq 1$. By Theorem \ref{TheoG1} and
given $\theta \in \left[ \frac{289}{240},\frac{17}{6}\right] $, $E_{0}=(%
\frac{3}{2},0)$ is globally asymptotically stable.
\end{itemize}

\begin{figure}[tbp]
\centering \includegraphics[width=5in]{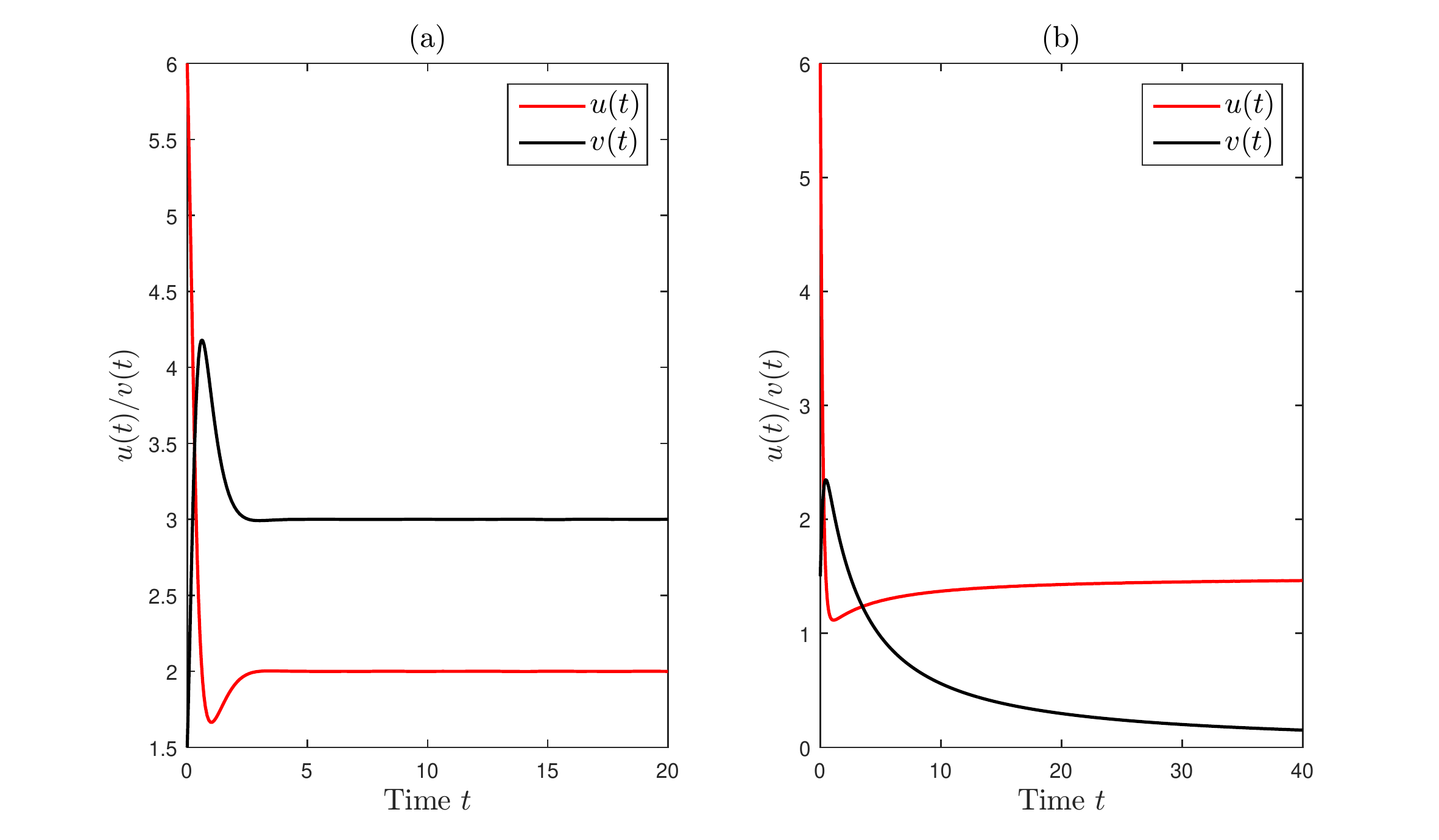}
\caption{Numerical solutions of system (\protect \ref{e1.25}) (ODE case)
subject to the first and second sets of parameters.}
\label{Fig1_ode}
\end{figure}

\begin{figure}[tbph]
\centering \includegraphics[width=5in]{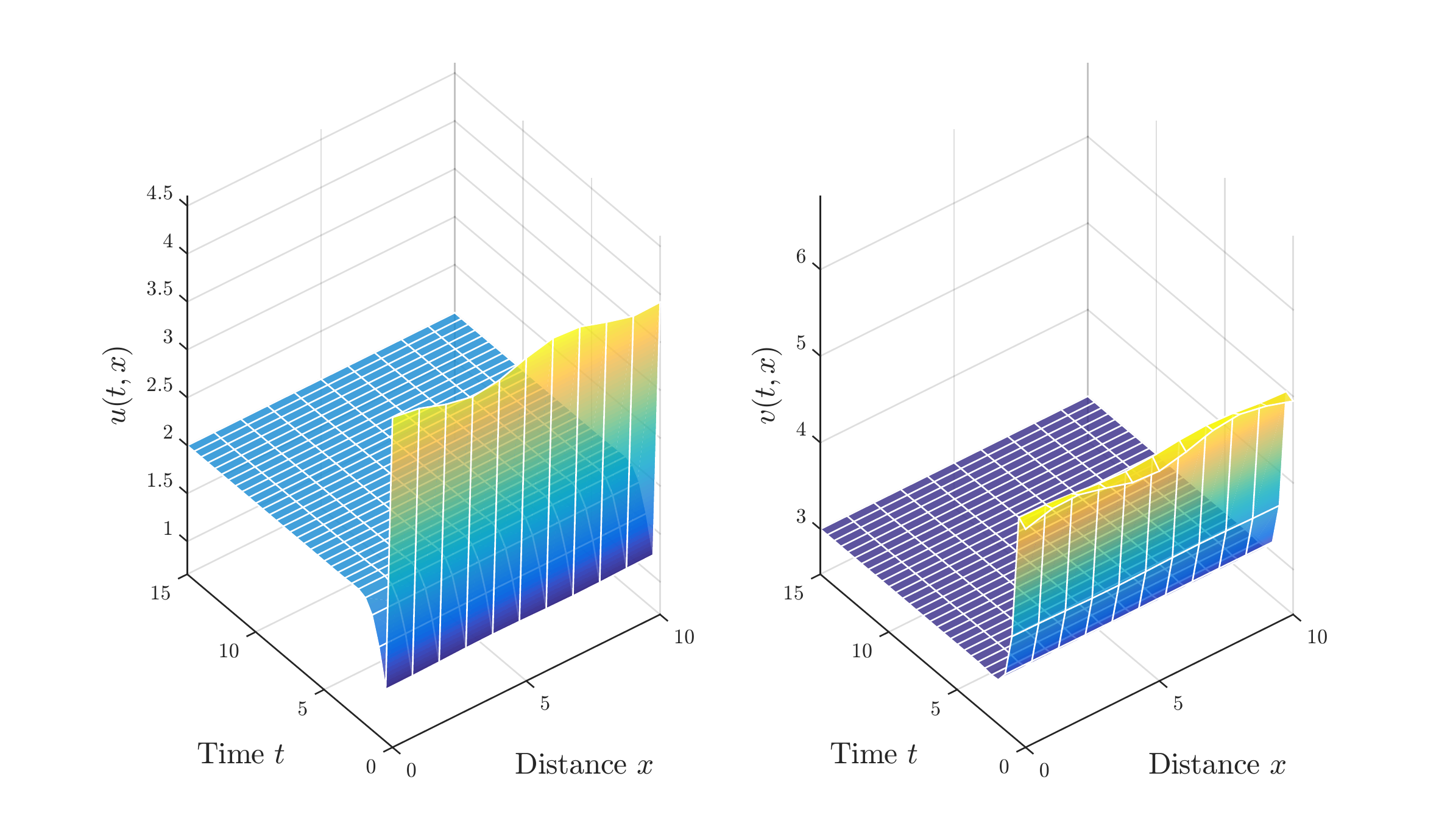}
\caption{Numerical solutions of system (\protect \ref{e1.25}) subject to the
first set of parameters.}
\label{Fig1}
\end{figure}

\begin{figure}[tbph]
\centering \includegraphics[width=5in]{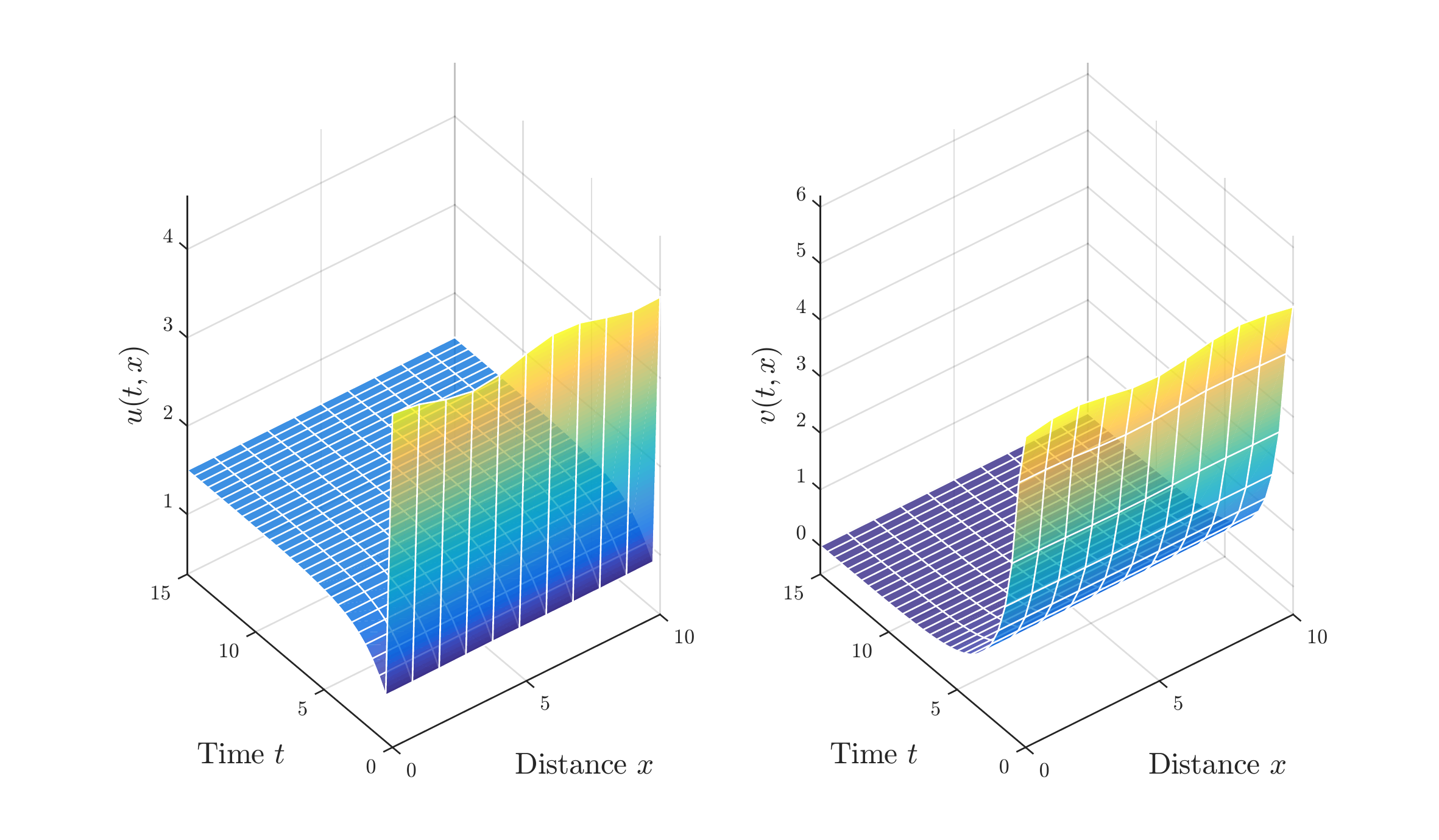}
\caption{Numerical solutions of system (\protect \ref{e1.25}) subject to the
second set of parameters.}
\label{Fig2}
\end{figure}

\subsection{Second Example}

The second numerical example which we are interested in is the PDE extension
of the ODE SIR\ model studied in \cite{Lahrouz2011}, which is a special case
of (\ref{e1.1}) with $\varphi \left( v\right) =\frac{\alpha v}{1+kv},\alpha
>0$ and $k\geq 0$. The resulting system is described by%
\begin{equation}
\left \{
\begin{array}{ll}
\dfrac{\partial u}{\partial t}-d_{1}\Delta u=-\lambda \frac{\alpha uv}{1+kv}%
+\Lambda -\mu u & \text{in }(0,\infty )\times \Omega , \\
\dfrac{\partial v}{\partial t}-d_{2}\Delta v=\lambda \frac{\alpha uv}{1+kv}%
-\sigma v & \text{in }(0,\infty )\times \Omega , \\
u(0,x)=u_{0}(x){,\  \ }v(0,x)=v_{0}(x) & \text{on }\Omega , \\
\frac{\partial u}{\partial \nu }=\frac{\partial v}{\partial \nu }=0, & \text{%
on }(0,\infty )\times \partial \Omega .%
\end{array}%
\right.   \label{e1.26}
\end{equation}%
The imposed conditions (\ref{e1.4}) and (\ref{e1.5}) can be easily verified.
It is evident that%
\begin{equation*}
\varphi \left( 0\right) =0\text{ and }\varphi \left( v\right) >0\text{ for
all }v>0.
\end{equation*}%
Also, the derivative of $\varphi \left( v\right) $ is given by
\begin{eqnarray*}
\varphi ^{\prime }\left( v\right)  &=&\left( \frac{\alpha v}{1+kv}\right)
^{\prime } \\
&=&\frac{\alpha }{\left( 1+kv\right) ^{2}}>0\text{ and }\varphi ^{\prime
}\left( 0\right) =\alpha .
\end{eqnarray*}%
In addition, we have%
\begin{equation*}
v\varphi ^{\prime }\left( v\right) =v\frac{\alpha }{\left( 1+kv\right) ^{2}}%
\leq \frac{\alpha v}{1+kv}=\varphi \left( v\right) .
\end{equation*}%
The constant steady states of system (\ref{e1.26}) are
\begin{equation*}
E_{0}=\left( \frac{\Lambda }{\mu },0\right) \  \  \text{and\  \ }E^{\ast
}=\left( \frac{\sigma \left( 1+kv^{\ast }\right) }{\lambda \alpha },\mu
\frac{\left( R_{0}-1\right) }{\left( \lambda \alpha +k\mu \right) }\right) .
\end{equation*}%
Again, $E^{\ast }$ exists and is globally asymptotically stable provided
that the reproduction number $R_{0}=\frac{\lambda \Lambda }{\mu \sigma }%
\alpha >1$. On the other hand, $E_{0}$ is globally asymptotically stable in
the ODE case with no conditions and in the PDE case if $\frac{%
(d_{1}+d_{2})^{2}}{4d_{1}d_{2}}\leqslant \theta \leqslant \frac{\mu }{%
\Lambda }\left( \frac{\mu +\sigma }{\lambda \alpha }-\frac{\Lambda }{\sigma }%
\right) $ when $d_{1}\neq d_{2}$, and if $\theta =1$ when $d_{1}=d_{2}$.
Table \ref{Tab2} shows the parameter sets used in the numerical simulations.

\begin{table}[tbp]
\caption{Simulation parameters for example 2: system (\protect \ref{e1.26}).}
\label{Tab2}\centering%
\begin{tabular}{|l|l|l|l|l|l|l|l|l|l|l|l|l|l|}
\hline
& Set & Figure & $u_{0}$ & $v_{0}$ & $d_{1}$ & $d_{2}$ & $\lambda $ & $%
\alpha $ & $\sigma $ & $\Lambda $ & $\mu $ & $k$ & $R_{0}$ \\ \hline
& Set 1 & \ref{Fig2_ode}(a) & $0.2$ & $4.3$ & $-$ & $-$ & $\frac{7}{12}$ & $%
\frac{13}{4}$ & $\frac{9}{4}$ & $\frac{33}{4}$ & $\frac{5}{4}$ & $\frac{1}{2}
$ & $5.5611$ \\ \cline{2-14}
ODE & Set 2 & \ref{Fig2_ode}(b) & $0.2$ & $4.3$ & $-$ & $-$ & $2$ & $\frac{1%
}{3}$ & $1$ & $5$ & $4$ & $7$ & $0.8333$ \\ \cline{2-14}
case & Set 3 & \ref{Fig3_ode}(a) & $8$ & $10$ & $-$ & $-$ & $\frac{7}{12}$ &
$\frac{13}{4}$ & $\frac{9}{4}$ & $\frac{33}{4}$ & $\frac{5}{4}$ & $\frac{1}{2%
}$ & $5.5611$ \\ \cline{2-14}
& Set 4 & \ref{Fig3_ode}(b) & $8$ & $10$ & $-$ & $-$ & $2$ & $\frac{1}{3}$ &
$1$ & $5$ & $4$ & $7$ & $0.8333$ \\ \hline
& Set 1 & \ref{Fig3} & $0.2+\frac{\cos (x)}{10}$ & $0.6+\frac{\sin (x)}{10}$
& $3$ & $2$ & $\frac{7}{12}$ & $\frac{13}{4}$ & $\frac{9}{4}$ & $\frac{33}{4}
$ & $\frac{5}{4}$ & $\frac{1}{2}$ & $5.5611$ \\ \cline{2-14}
PDE & Set 2 & \ref{Fig5} & $4+\frac{\cos (x)}{10}$ & $5+\frac{\sin (x)}{10}$
& $3$ & $2$ & $\frac{7}{12}$ & $\frac{13}{4}$ & $\frac{9}{4}$ & $\frac{33}{4}
$ & $\frac{5}{4}$ & $3$ & $5.5611$ \\ \cline{2-14}
case & Set 3 & \ref{Fig6} & $0.2+\frac{\cos (x)}{10}$ & $0.6+\frac{\sin (x)}{%
10}$ & $3$ & $2$ & $2$ & $\frac{1}{3}$ & $1$ & $5$ & $4$ & $\frac{2}{3}$ & $%
0.8333$ \\ \cline{2-14}
& Set 4 & \ref{Fig7} & $0.2+\frac{\cos (x)}{10}$ & $0.6+\frac{\sin (x)}{10}$
& $\frac{7}{2}$ & $\frac{5}{4}$ & $2$ & $\frac{1}{3}$ & $1$ & $5$ & $4$ & $7$
& $0.8333$ \\ \hline
\end{tabular}%
\end{table}

The following is a description of the results:

\begin{itemize}
\item Figures \ref{Fig2_ode} and \ref{Fig3_ode} depict the numerical
solutions obtained using the four parameter sets in the ODE case with steady
states $E^{\ast }=(2.5289,2.2617)$, $E_{0}=(1.25,0)$, $E^{\ast
}=(2.5289,2.2617)$, and $E_{0}=(1.25,0)$, respectively. In all four
scenarios, both the analytical and numerical methods agree that the
equilibria are stable.

\item Figure \ref{Fig3} shows the PDE solution obtained using parameter set
1 with $E^{\ast }=(2.5289,2.2617)$. In this case, $R_{0}=5.5611>1$ and by
Theorem \ref{TheoG2}, $E^{\ast }$ is globally asymptotically stable.

\item Figure \ref{Fig5} shows the PDE solution obtained using parameter set
2 with $E^{\ast }=(4.7823,1.0098)$. Since $R_{0}=5.5611>1$, $E^{\ast }$ is
globally asymptotically stable.

\item Figure \ref{Fig6} shows the PDE solution obtained using parameter set
3 with $E_{0}=(1.25,0)$. In this case, $R_{0}=0.8333\leq 1$ and using
Theorem \ref{TheoG1} with $\theta \in \left[ \frac{25}{24},2\right] $, $%
E_{0} $ is globally asymptotically stable.

\item Figure \ref{Fig7} shows the PDE solution obtained using parameter set
4 with $E_{0}=(1.25,0)$. In this scenario, we again have $R_{0}=0.8333\leq 1$
and with $\theta \in \left[ \frac{361}{280},2\right] $, $E_{0}$ is globally
asymptotically stable.
\end{itemize}

\begin{figure}[tbp]
\centering
\includegraphics[width=5in]{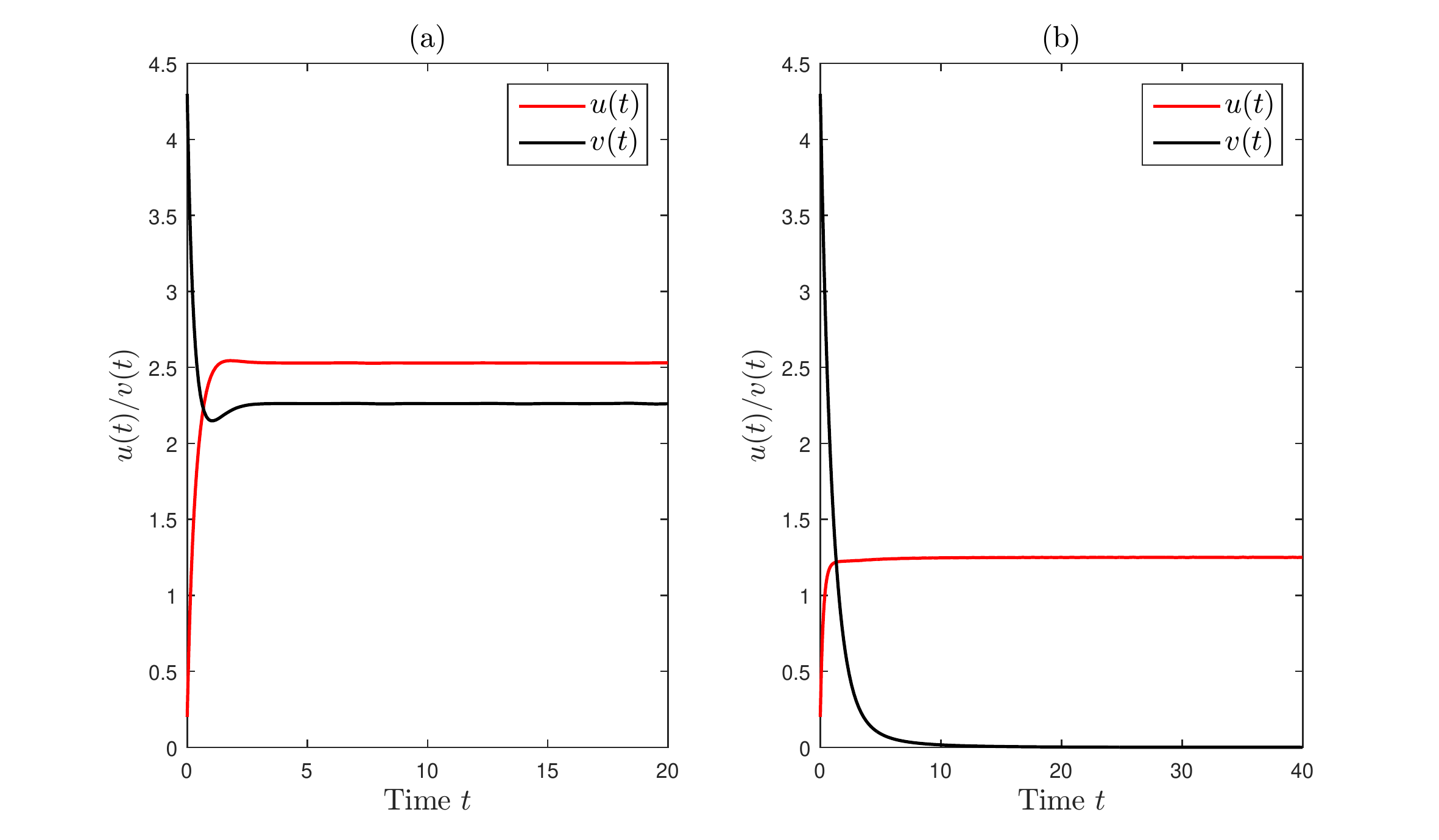}
\caption{Numerical solutions of system (\protect \ref{e1.26}) (ODE case)
subject to the first and second sets of parameters.}
\label{Fig2_ode}
\end{figure}

\begin{figure}[tbp]
\centering
\includegraphics[width=5in]{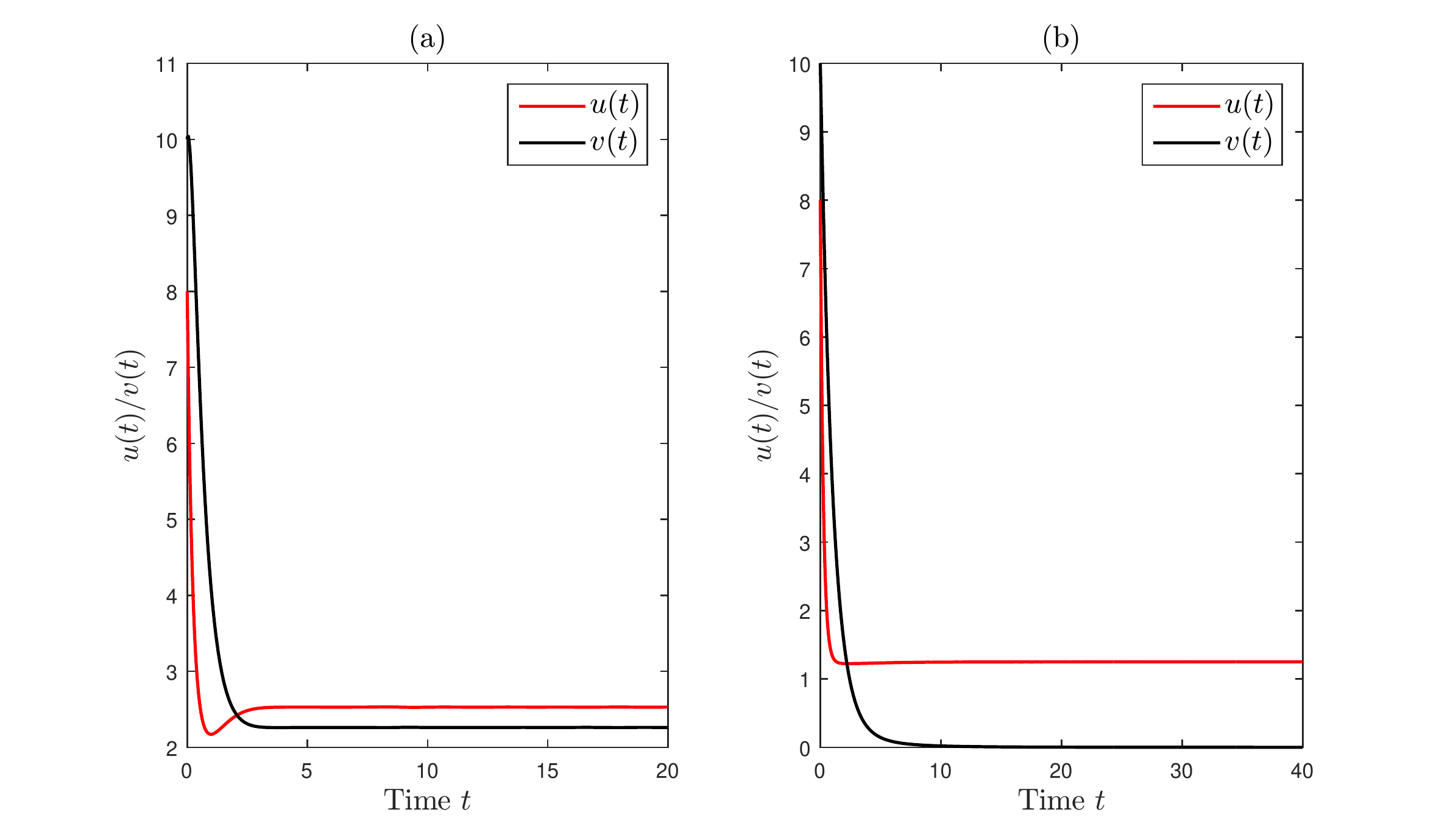}
\caption{Numerical solutions of system (\protect \ref{e1.26}) (ODE case)
subject to the third and fourth sets of parameters.}
\label{Fig3_ode}
\end{figure}

\begin{figure}[tbph]
\centering
\includegraphics[width=5in]{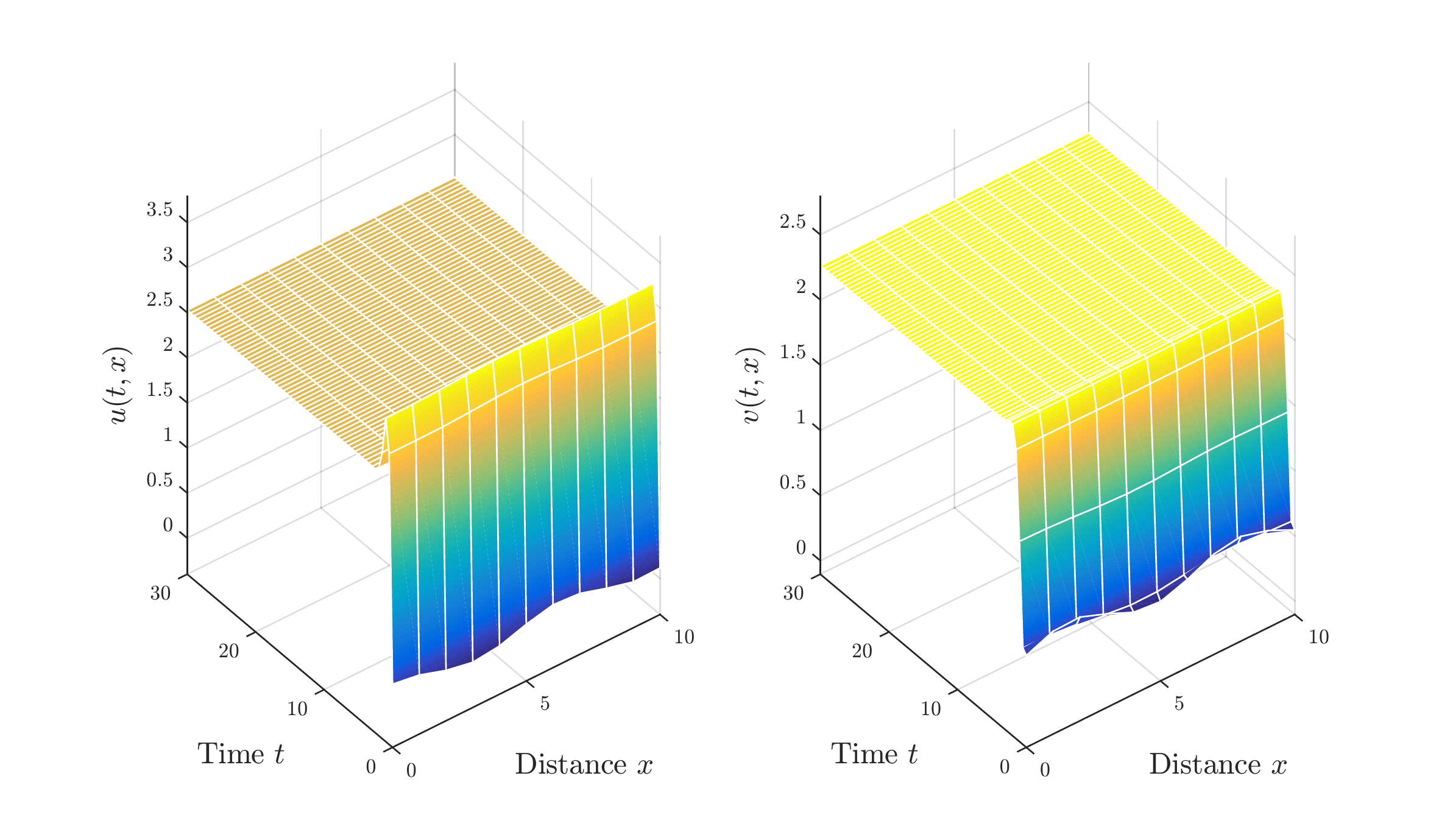}
\caption{Numerical solutions of system (\protect \ref{e1.26}) subject to the
first set of parameters.}
\label{Fig3}
\end{figure}

\begin{figure}[tbph]
\centering
\includegraphics[width=5in]{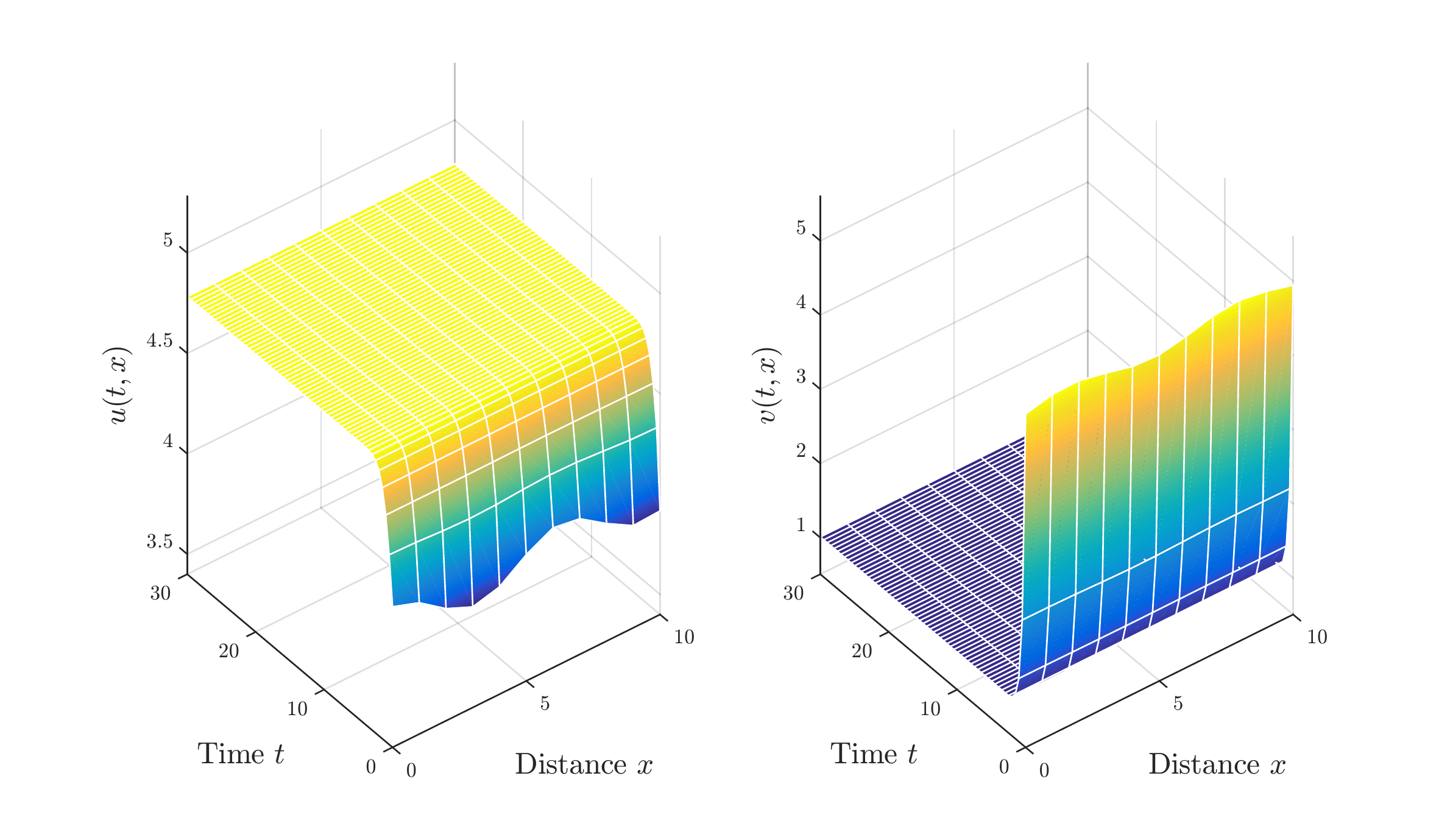}
\caption{Numerical solutions of system (\protect \ref{e1.26}) subject to the
second set of parameters.}
\label{Fig5}
\end{figure}

\begin{figure}[tbph]
\centering
\includegraphics[width=5in]{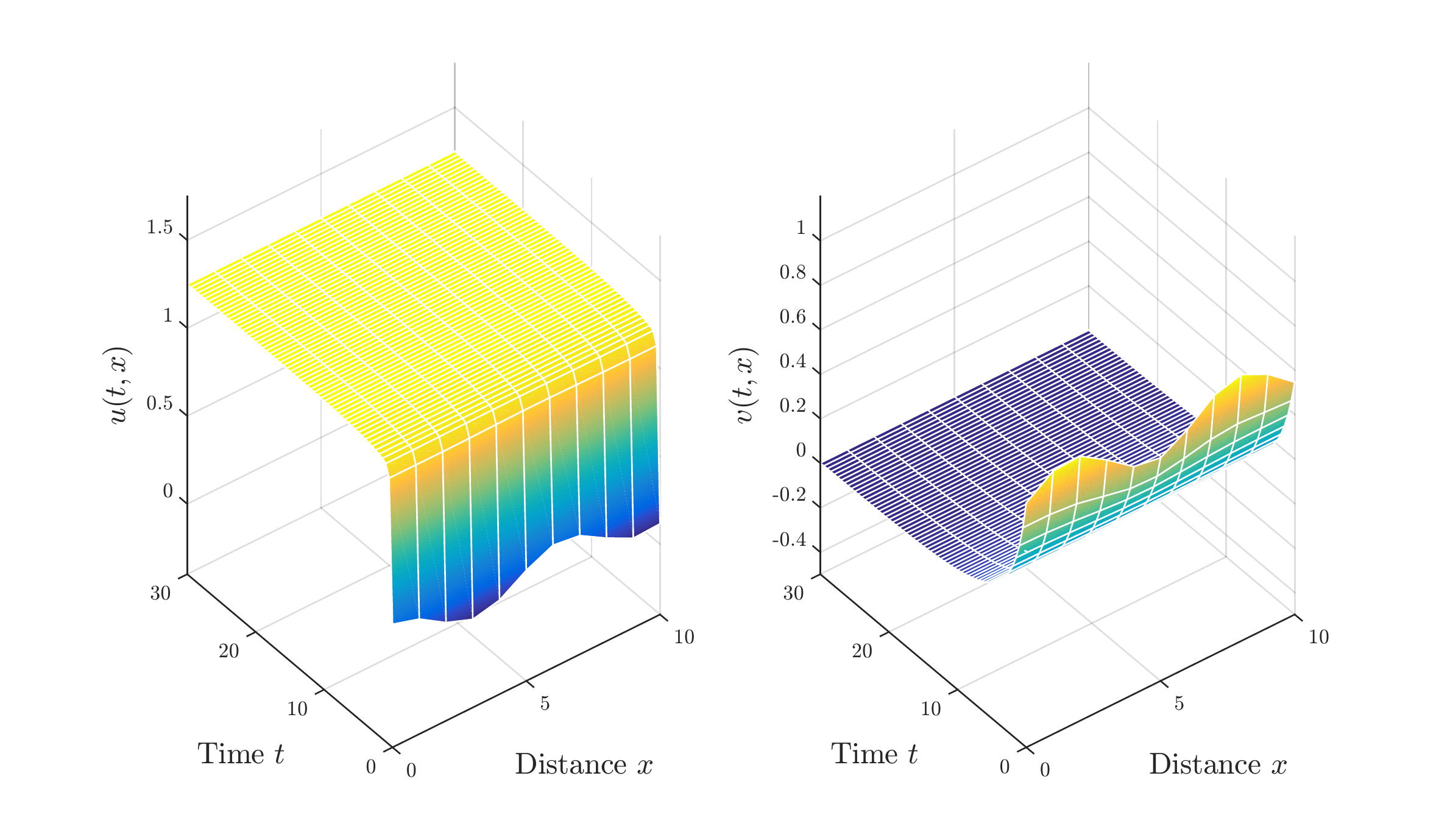}
\caption{Numerical solutions of system (\protect \ref{e1.26}) subject to the
third set of parameters.}
\label{Fig6}
\end{figure}

\begin{figure}[tbph]
\centering
\includegraphics[width=5in]{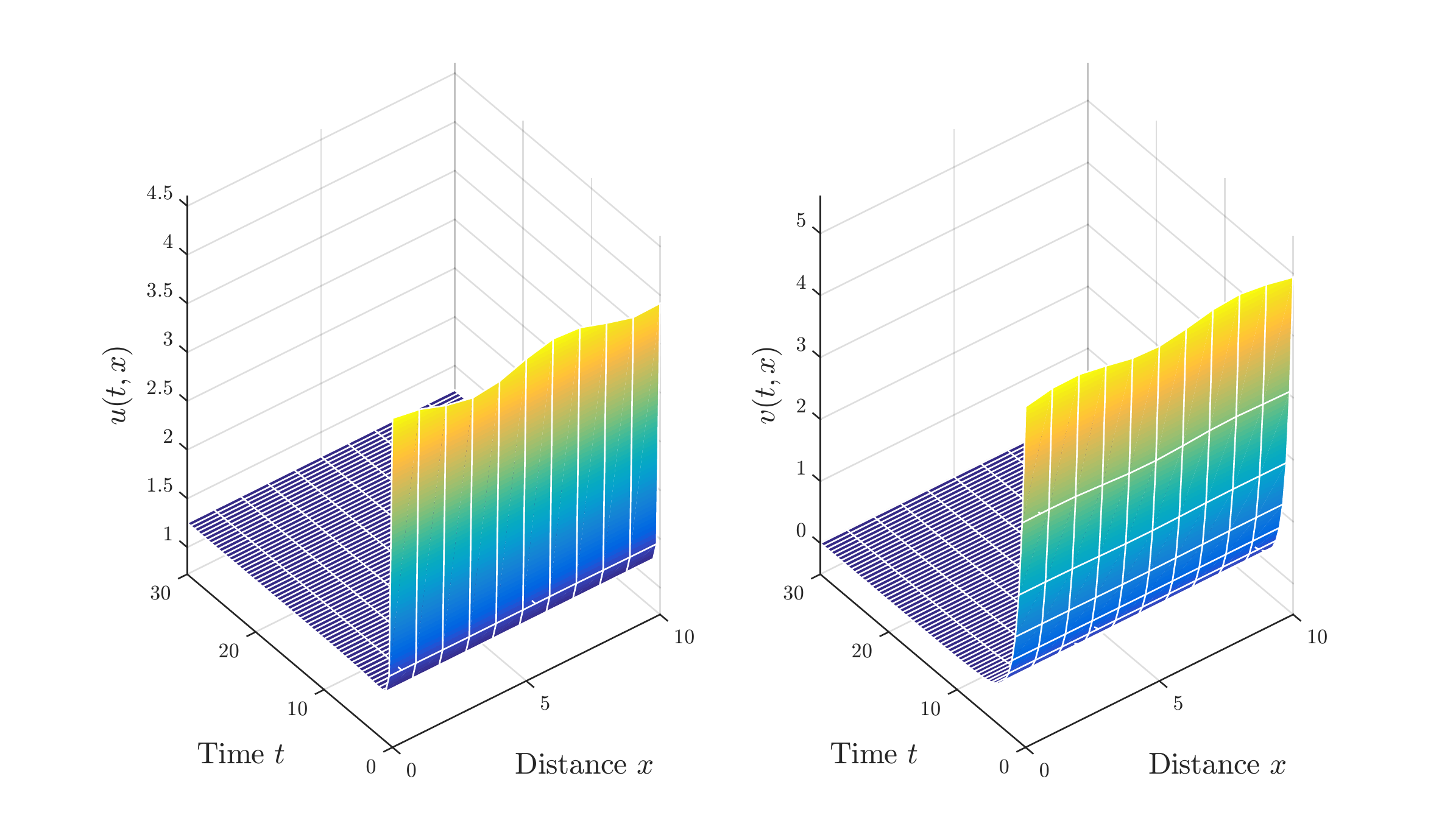}
\caption{Numerical solutions of system (\protect \ref{e1.26}) subject to the
fourth set of parameters.}
\label{Fig7}
\end{figure}

\subsection{Third Example}

The last example is obtained by substituting $\varphi \left( v\right) =\frac{%
kv}{1+\left( \frac{v}{\alpha }\right) }$, which yields the same system
studied in \cite{Capasso1978,Li2009} but with $d_{1}=d_{2}=0$. The resulting
system is given by%
\begin{equation}
\left \{
\begin{array}{ll}
\dfrac{\partial u}{\partial t}-d_{1}\Delta u=-\lambda k\frac{v}{1+\left(
\frac{v}{\alpha }\right) }u+\Lambda -\mu u & \text{in }(0,\infty )\times
\Omega , \\
\dfrac{\partial v}{\partial t}-d_{2}\Delta v=\lambda k\frac{v}{1+\left(
\frac{v}{\alpha }\right) }u-\sigma v & \text{in }(0,\infty )\times \Omega ,
\\
u(0,x)=u_{0}(x){,\  \ }v(0,x)=v_{0}(x) & \text{on }\Omega , \\
\frac{\partial u}{\partial \nu }=\frac{\partial v}{\partial \nu }=0, & \text{%
on }(0,\infty )\times \partial \Omega ,%
\end{array}%
\right.   \label{e1.27}
\end{equation}%
for $\alpha >0$ and $k>0$. The imposed conditions may be verified as%
\begin{equation*}
\left \{
\begin{array}{l}
\varphi \left( 0\right) =0, \\
\varphi ^{\prime }\left( v\right) >0\  \text{for all }v>0, \\
\varphi ^{\prime }\left( 0\right) =\frac{k}{\left( 1+\left( \frac{v}{\alpha }%
\right) \right) ^{2}}>0, \\
v\varphi ^{\prime }\left( v\right) =v\frac{k}{\left( 1+\left( \frac{v}{%
\alpha }\right) \right) ^{2}}\leq \frac{kv}{1+\left( \frac{v}{\alpha }%
\right) }=\varphi \left( v\right) .%
\end{array}%
\right.
\end{equation*}%
The steady states of system (\ref{e1.27}) are given by $E_{0}=\left( \frac{%
\Lambda }{\mu },0\right) $ and $E^{\ast }=\left( \frac{\sigma \left( \alpha
+v^{\ast }\right) }{\lambda \alpha k},\mu \alpha \frac{\left( R_{0}-1\right)
}{\left( \lambda \alpha k+\mu \right) }\right) $ with\ the reproduction
number $R_{0}=\frac{\lambda \Lambda }{\mu \sigma }k>1$. Note that if $%
E^{\ast }$ exists than it is globally asymptotically stable and that $E_{0}$%
\ is globally asymptotically stable if $\frac{(d_{1}+d_{2})^{2}}{4d_{1}d_{2}}%
\leqslant \theta \leqslant \frac{\mu }{\Lambda }\left( \frac{\mu +\sigma }{%
\lambda k}-\frac{\Lambda }{\sigma }\right) $ when $d_{1}\neq d_{2}$, and if $%
\theta =1$ when $d_{1}=d_{2}$ or in the ODE case. Table \ref{Tab3} details
the sets of parameters used in the numerical simulations.

\begin{table}[tbp]
\caption{Simulation parameters for example 2: system (\protect \ref{e1.26}).}
\label{Tab3}\centering%
\begin{tabular}{|l|l|l|l|l|l|l|l|l|l|l|l|l|l|}
\hline
& Set & Figure & $u_{0}$ & $v_{0}$ & $d_{1}$ & $d_{2}$ & $\lambda $ & $%
\alpha $ & $\sigma $ & $\Lambda $ & $\mu $ & $k$ & $R_{0}$ \\ \hline
& Set 1 & \ref{Fig4_ode}(a) & $0.8$ & $1.2$ & $-$ & $-$ & $1$ & $2$ & $3$ & $%
6$ & $\frac{1}{3}$ & $2$ & $12$ \\ \cline{2-14}
ODE & Set 2 & \ref{Fig4_ode}(b) & $0.4$ & $6$ & $-$ & $-$ & $\frac{1}{2}$ & $%
1$ & $2$ & $\frac{3}{4}$ & $\frac{3}{7}$ & $\frac{4}{3}$ & $0.5833$ \\
\cline{2-14}
case & Set 3 & \ref{Fig5_ode}(a) & $0.2$ & $4$ & $-$ & $-$ & $1$ & $2$ & $3$
& $8$ & $\frac{2}{3}$ & $2$ & $8$ \\ \cline{2-14}
& Set 4 & \ref{Fig5_ode}(b) & $0.2$ & $3$ & $-$ & $-$ & $\frac{1}{2}$ & $1$
& $2$ & $\frac{3}{5}$ & $\frac{3}{7}$ & $\frac{6}{5}$ & $0.42$ \\ \hline
& Set 1 & \ref{Fig8} & $4+\frac{\cos (x)}{10}$ & $5+\frac{\sin (x)}{10}$ & $%
3 $ & $\frac{5}{4}$ & $1$ & $2$ & $3$ & $6$ & $\frac{1}{3}$ & $2$ & $12$ \\
\cline{2-14}
PDE & Set 2 & \ref{Fig9} & $0.6+\frac{\cos (x)}{7}$ & $0.4+\frac{\sin (x)}{8}
$ & $5$ & $2$ & $1$ & $2$ & $3$ & $6$ & $\frac{1}{3}$ & $2$ & $12$ \\
\cline{2-14}
case & Set 3 & \ref{Fig10} & $2.6+\frac{\cos (x)}{7}$ & $2.4+\frac{\sin (x)}{%
8}$ & $2$ & $1$ & $1$ & $2$ & $3$ & $8$ & $\frac{2}{3}$ & $2$ & $8$ \\
\cline{2-14}
& Set 4 & \ref{Fig11} & $4+\frac{\cos (x)}{10}$ & $5+\frac{\sin (x)}{10}$ & $%
3$ & $\frac{5}{4}$ & $\frac{1}{2}$ & $1$ & $2$ & $\frac{3}{4}$ & $\frac{3}{7}
$ & $\frac{4}{3}$ & $0.5833$ \\ \cline{2-14}
& Set 5 & \ref{Fig12} & $0.6+\frac{\cos (x)}{7}$ & $0.4+\frac{\sin (x)}{8}$
& $\frac{13}{4}$ & $2$ & $\frac{1}{2}$ & $1$ & $2$ & $\frac{3}{5}$ & $\frac{3%
}{7}$ & $\frac{6}{5}$ & $0.42$ \\ \hline
\end{tabular}%
\end{table}

The following is a description of the results:

\begin{itemize}
\item Figures \ref{Fig4_ode} and \ref{Fig5_ode} show the numerical solutions
of system (\ref{e1.27}) resulting the four parameter sets in the ODE case
with steady states $E^{\ast }=(2.7692,1.6923)$, $E_{0}=(1.75,0)$, $E^{\ast
}=(3,2)$, and $E_{0}=(1.4,0)$, respectively. In all four scenarios, both the
analytical and numerical methods agree that the equilibria are stable.

\item Figure \ref{Fig8} shows the PDE solution obtained using parameter set
1 with $E^{\ast }=(2.7692,1.6923)$. In this case, $R_{0}=12>1$ and by
Theorem \ref{TheoG2}, $E^{\ast }$ is globally asymptotically stable.

\item Figure \ref{Fig9} shows the PDE solution obtained using parameter set
2 with $E^{\ast }=(2.7692,1.6923)$. In this case, $R_{0}=12>1$ and $E^{\ast
} $ is globally asymptotically stable.

\item Figure \ref{Fig10} shows the PDE solution obtained using parameter set
3 with $E^{\ast }=(3,2)$. In this case, $R_{0}=8>1$ and $E^{\ast }$ is
globally asymptotically stable.

\item Figure \ref{Fig11} shows the PDE solution obtained using parameter set
4 with $E_{0}=(1.75,0)$. In this case, $R_{0}=0.5833\leq 1$ and by Theorem %
\ref{TheoG1} with $\theta \in \left[ \frac{289}{240},\frac{394}{211}\right] $%
, $E_{0}$ is globally asymptotically stable.

\item Figure \ref{Fig12} shows the PDE solution obtained using parameter set
5 with $E_{0}=(1.4,0)$. In this case, $R_{0}=0.42\leq 1$ and by Theorem \ref%
{TheoG1} with $\theta \in \left[ \frac{441}{416},\frac{1831}{684}\right] $, $%
E_{0}$ is globally asymptotically stable.
\end{itemize}

\begin{figure}[tbp]
\centering
\includegraphics[width=5in]{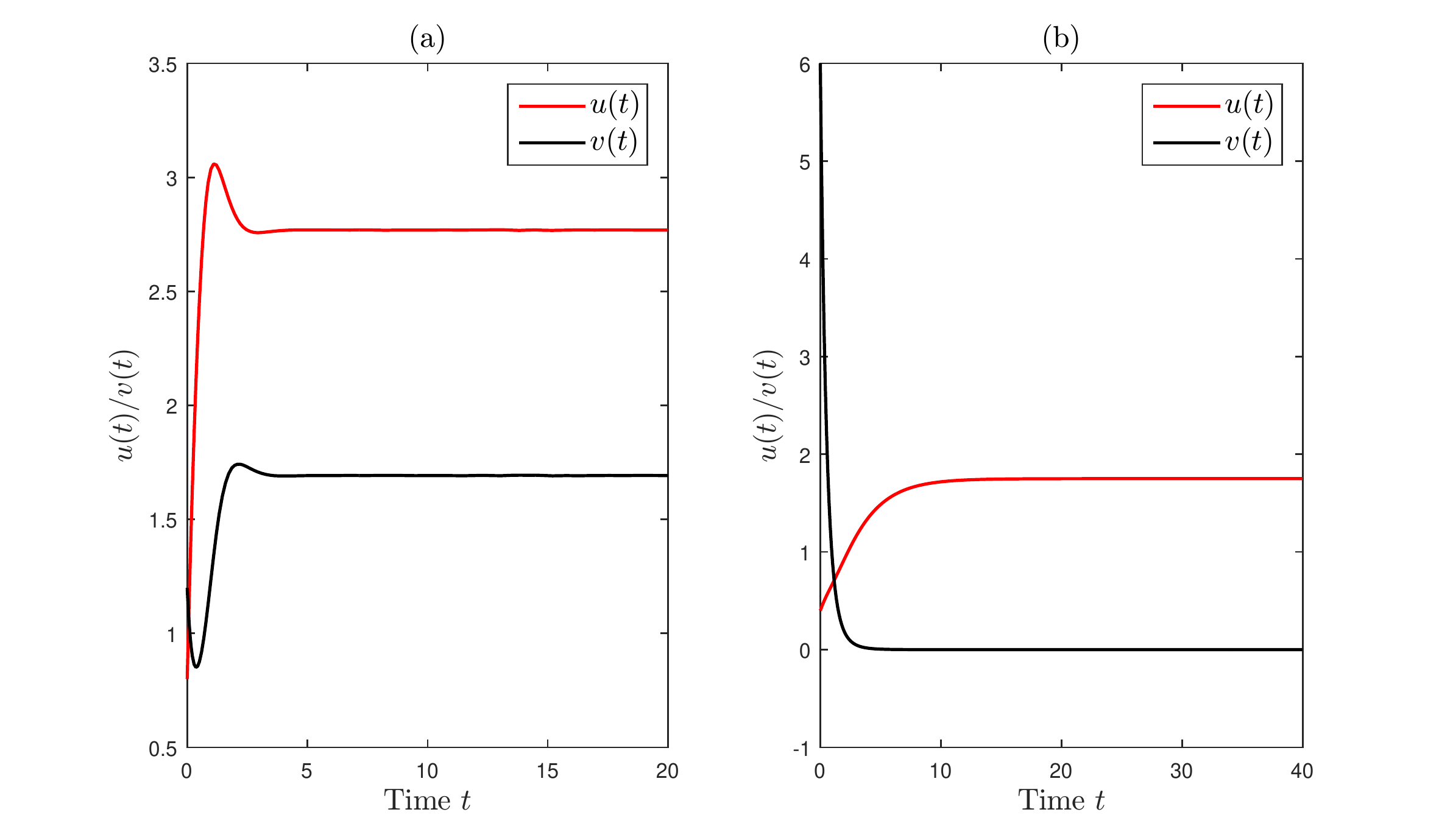}
\caption{Numerical solutions of system (\protect \ref{e1.27}) (ODE case)
subject to the first and second sets of parameters.}
\label{Fig4_ode}
\end{figure}

\begin{figure}[tbp]
\centering
\includegraphics[width=5in]{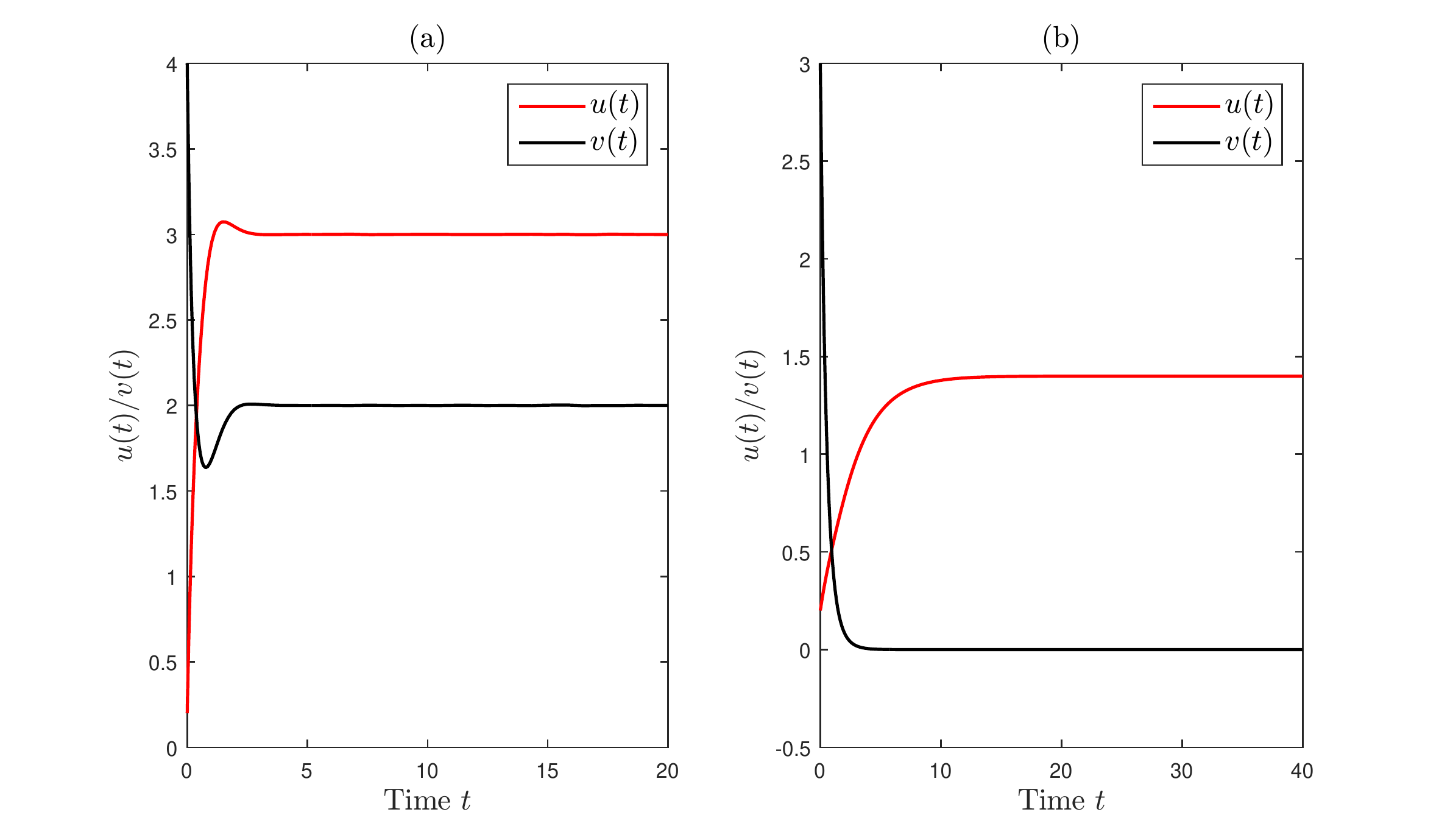}
\caption{Numerical solutions of system (\protect \ref{e1.27}) (ODE case)
subject to the third and fourth sets of parameters.}
\label{Fig5_ode}
\end{figure}

\begin{figure}[tbph]
\centering
\includegraphics[width=5in]{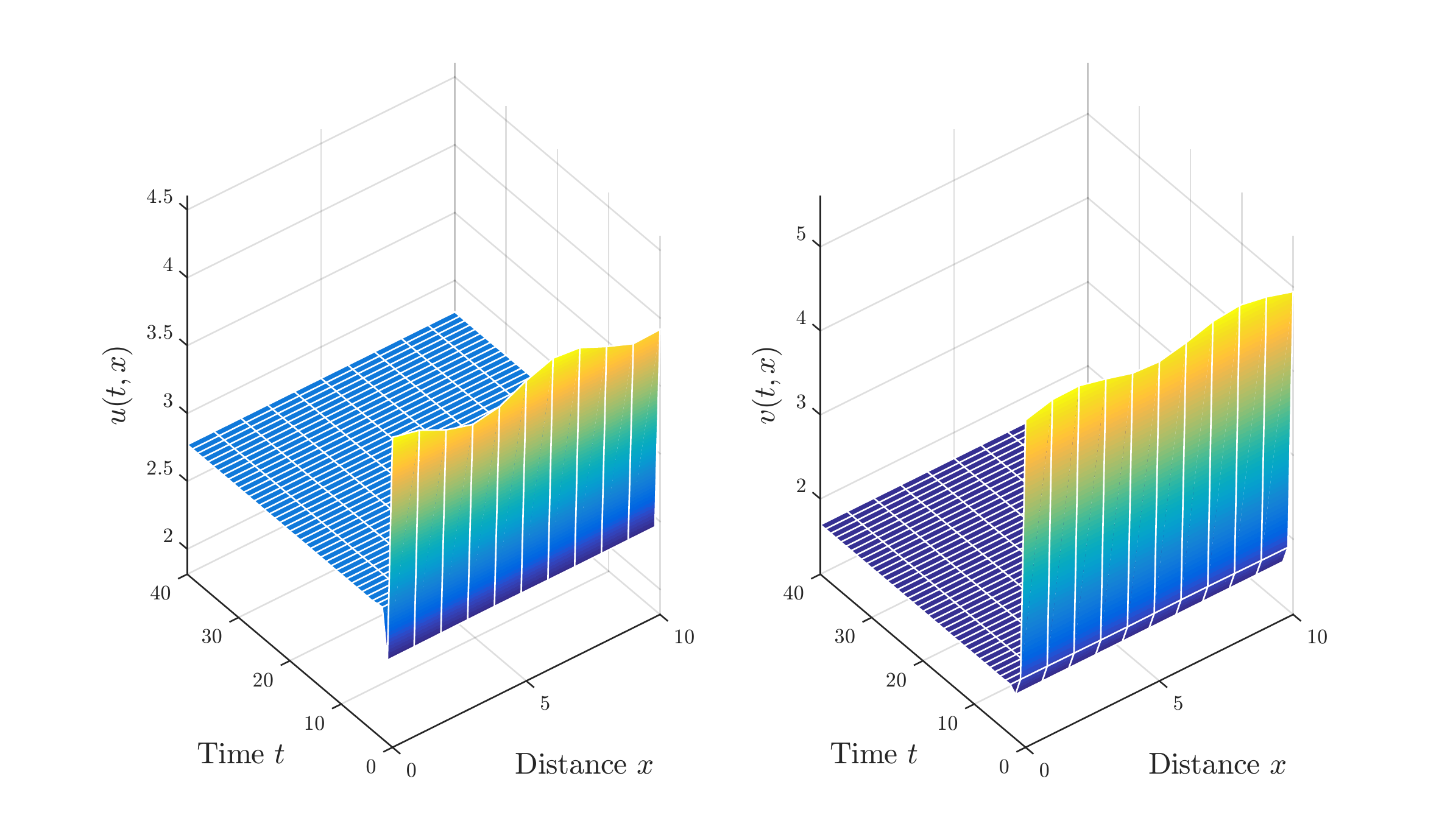}
\caption{Numerical solutions of system (\protect \ref{e1.27}) subject to the
first set of parameters.}
\label{Fig8}
\end{figure}

\begin{figure}[tbph]
\centering
\includegraphics[width=5in]{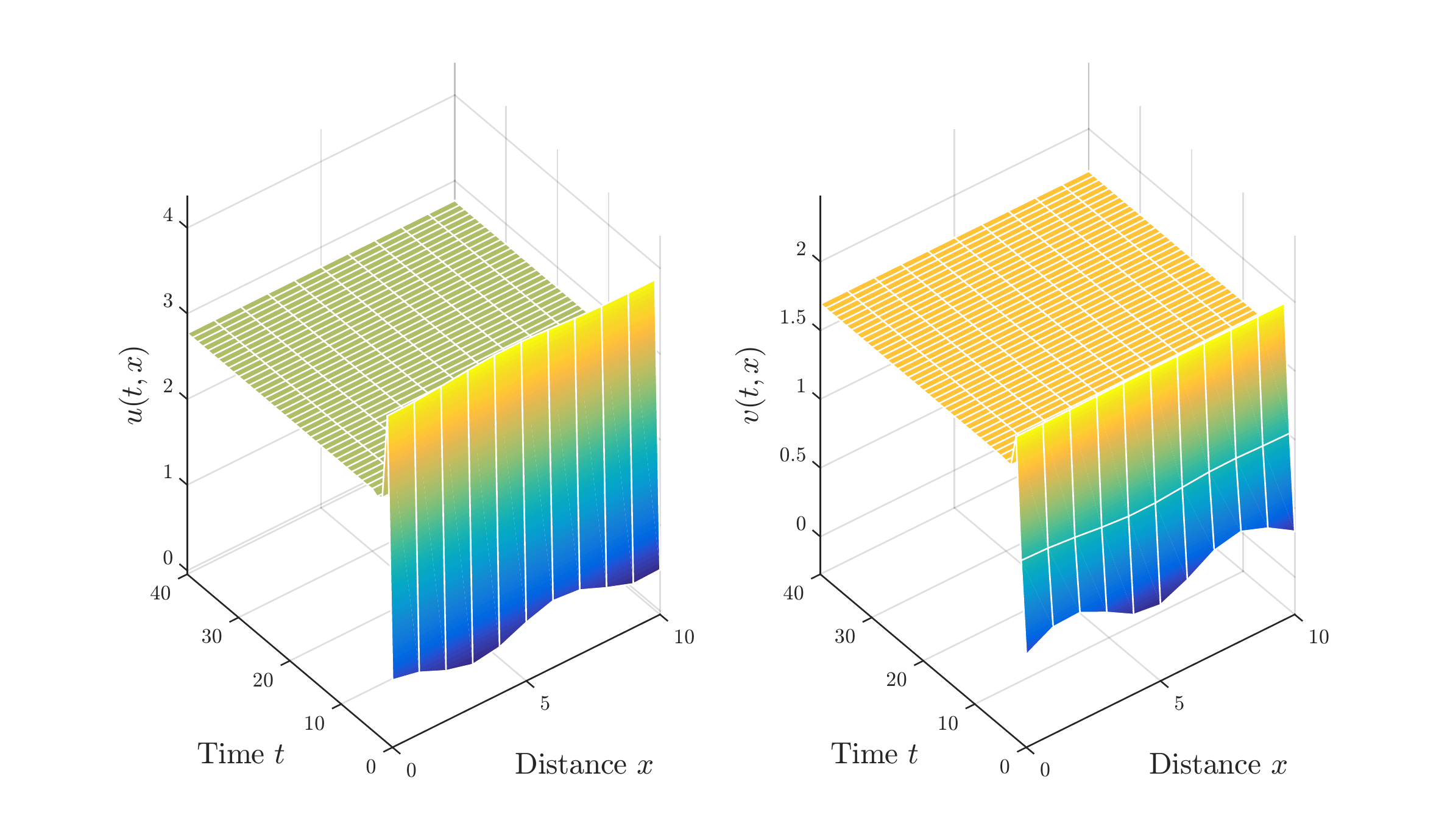}
\caption{Numerical solutions of system (\protect \ref{e1.27}) subject to the
second set of parameters.}
\label{Fig9}
\end{figure}

\begin{figure}[tbph]
\centering
\includegraphics[width=5in]{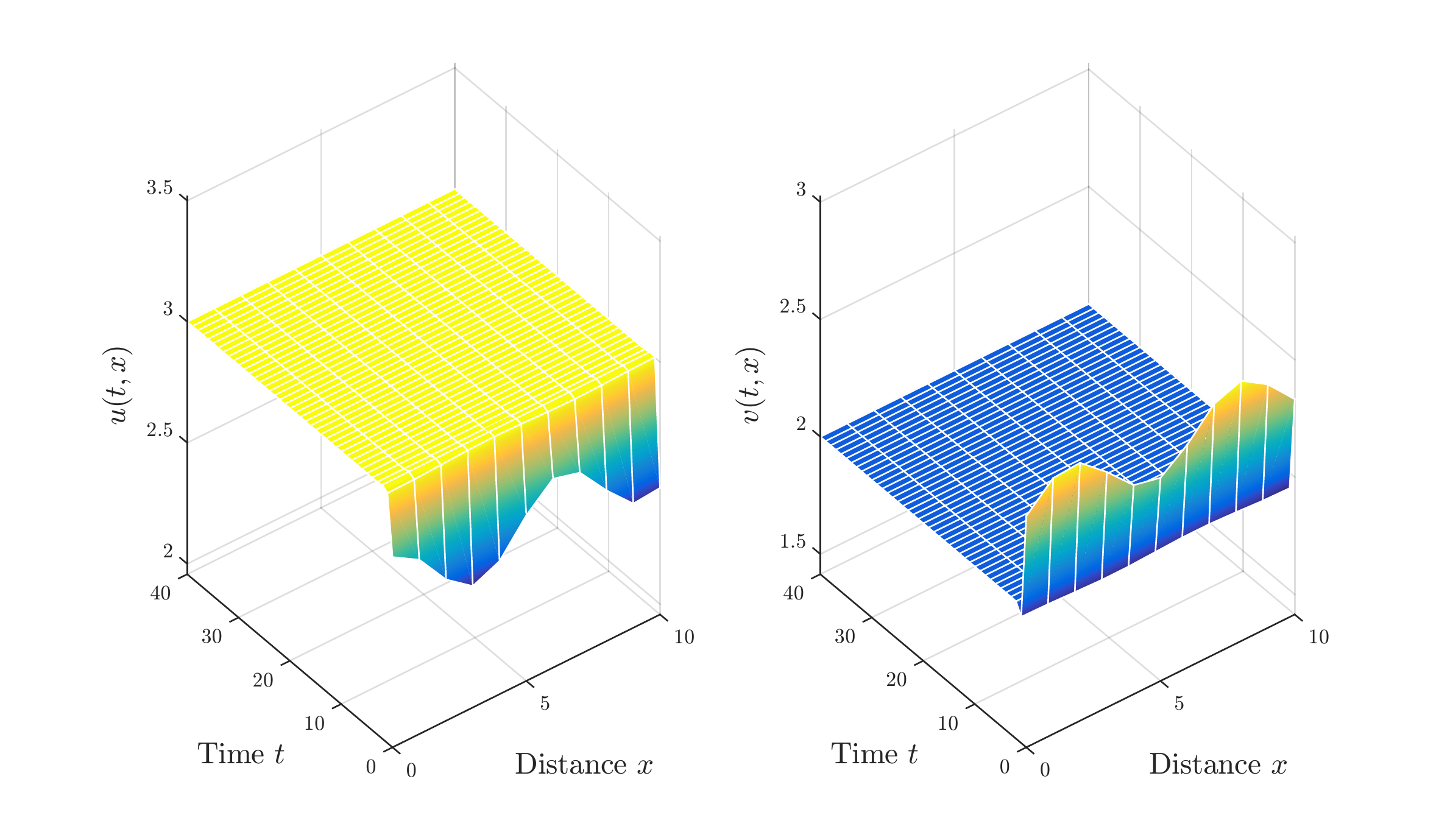}
\caption{Numerical solutions of system (\protect \ref{e1.27}) subject to the
third set of parameters.}
\label{Fig10}
\end{figure}

\begin{figure}[tbph]
\centering
\includegraphics[width=5in]{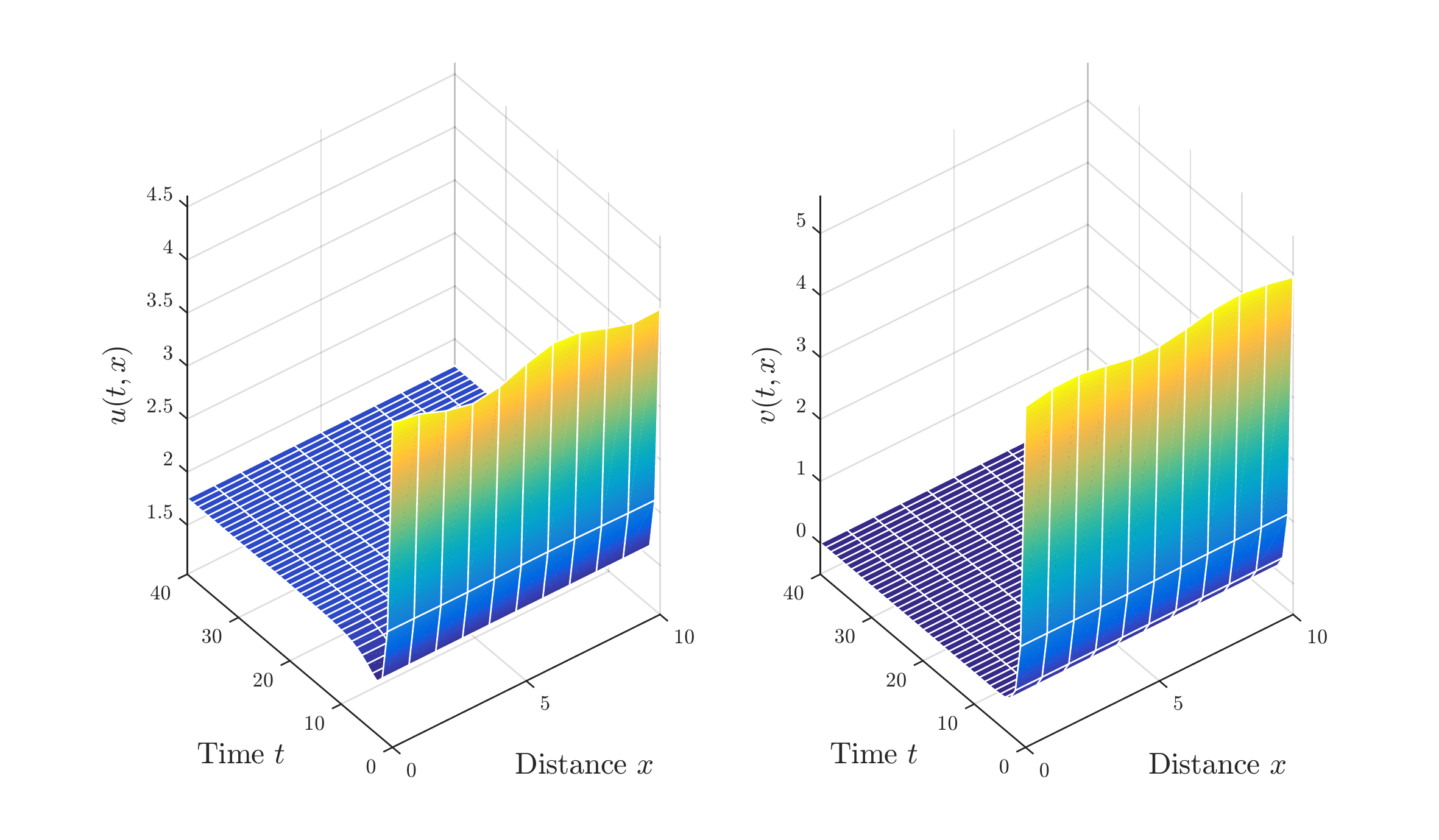}
\caption{Numerical solutions of system (\protect \ref{e1.27}) subject to the
forth set of parameters.}
\label{Fig11}
\end{figure}

\begin{figure}[tbph]
\centering
\includegraphics[width=5in]{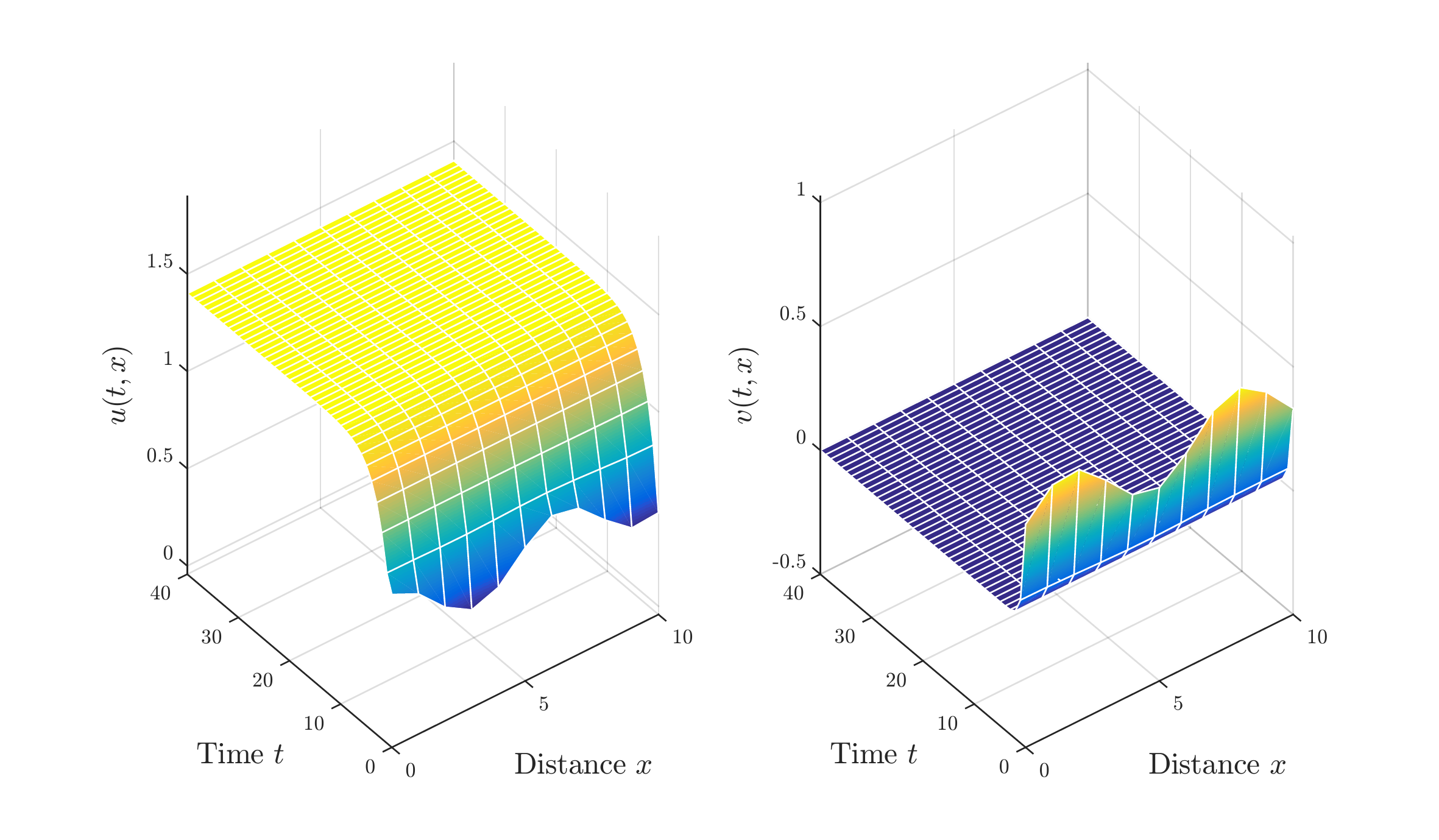}
\caption{Numerical solutions of system (\protect \ref{e1.27}) subject to the
fifth set of parameters.}
\label{Fig12}
\end{figure}

\end{document}